\documentclass[12pt]{amsart}

\usepackage{tikz}
\usetikzlibrary{arrows, shapes}
\usepackage{xcolor}
\usepackage[latin1]{inputenc}
\usepackage[all]{xy}
\usepackage{lscape}
\usepackage{amsmath}
\usepackage{amssymb}
\usepackage{amsthm}
\usepackage{pdfsync}
\usepackage[percent]{overpic}
\usepackage[linktocpage=true]{hyperref}
\usepackage{youngtab}

\numberwithin{equation}{section}

\makeatletter
\def\l@subsection{\@tocline{2}{0pt}{2.5pc}{5pc}{}}
\makeatother

\newcommand{\stb}{,\ldots,}
\newcommand{\then}{\Rightarrow}

\makeatletter
\DeclareRobustCommand{\bigtimes}{%
	\mathop{\vphantom{\sum}\mathpalette\@bigtimes\relax}\slimits@}
\newcommand{\@bigtimes}[2]{\vcenter{\hbox{\make@bigtimes{#1}}}}
\newcommand{\make@bigtimes}[1]{%
	\sbox\z@{$\m@th#1\sum$}%
	\setlength{\unitlength}{\wd\z@}%
	\begin{picture}(1,1)
	\linethickness{.15ex}
	\Line(0.2,0.2)(0.8,0.8)
	\Line(0.2,0.8)(0.8,0.2)
	\end{picture}}

\newcommand{\Tor}{\operatorname{Tor}}

\renewcommand{\mod}{\operatorname{mod}}

\newcommand{\GL}{\operatorname{GL}}

\newcommand{\SO}{\operatorname{SO}}

\newcommand{\U}{\operatorname{U}}

\newcommand{\Fl}{\operatorname{Fl}}
\newcommand{\fd}{\operatorname{fd}}

\newcommand{\rk}{\operatorname{rk}}
\newcommand{\Hom}{\operatorname{Hom}}

\newcommand{\End}{\operatorname{End}}
\newcommand{\iso}{\cong}
\newcommand{\isoto}{\xrightarrow{\iso}}

\newcommand{\I}{I}
\newcommand{\J}{J}
\newcommand{\sii}{\boldsymbol \si}
\newcommand{\Omm}{\boldsymbol \Om}

\newcommand{\im}{\operatorname{Im}}
\newcommand{\acsa}{\qquad \iff\qquad}
\newcommand{\id}{\operatorname{id}}
\newcommand{\OSP}{\operatorname{OSP}}
\newcommand{\bra}{\langle}

\newcommand{\ket}{\rangle}

\newcommand*\clos[1]{\overline{#1}}

\newcommand{\codim}{\operatorname{codim}}

\newcommand{\Sq}{\operatorname{Sq}}
\newcommand{\inj}{\hookrightarrow}
\newcommand{\RP}{\mathbb{R}P}

\newcommand{\PP}{\mathbb{P}}
\newcommand{\Gr}{\operatorname{Gr}}

\newcommand{\al}{\alpha}
\newcommand{\be}{\beta}
\newcommand{\ga}{\gamma}
\newcommand{\Ga}{\Gamma}

\newcommand{\de}{\delta}
\newcommand{\ep}{\varepsilon}

\newcommand{\ka}{\kappa}
\newcommand{\la}{\lambda}
\newcommand{\La}{\Lambda}
\newcommand{\si}{\sigma}
\newcommand{\Si}{\Sigma}
\newcommand{\Om}{\Omega}

\newcommand{\Stab}{\operatorname{Stab}}

\newcommand{\se}{\subseteq}
\newcommand{\su}{\backslash}

\newcommand{\Z}{\mathbb{Z}}
\newcommand{\R}{\mathbb{R}}
\newcommand{\C}{\mathbb{C}}
\newcommand{\N}{\mathbb{N}}
\newcommand{\Q}{\mathbb{Q}}
\newcommand{\F}{\mathbb{F}}

\newcommand{\D}{\mathcal{D}}

\renewcommand{\S}{\mathcal{S}}

\renewcommand{\aa}{\mathfrak{a}}
\renewcommand{\sl}{\mathfrak{s}\mathfrak{l}}

\renewcommand{\gg}{\mathfrak{g}}

\newcommand{\hh}{\mathfrak{h}}

\newcommand{\K}{\mathcal{K}}

\newtheorem{fact}{Fact}[section]
\newtheorem{lemma}[fact]{Lemma}
\newtheorem{theorem}[fact]{Theorem}
\newtheorem*{theorem*}{Theorem}
\newtheorem{defi}[fact]{Definition}
\newtheorem{exa}[fact]{Example}
\newtheorem{cla}[fact]{Claim}
\newtheorem*{sol}{\it Solution}
\newtheorem{rremark}[fact]{Remark}
\newtheorem{proposition}[fact]{Proposition}
\newtheorem{corollary}[fact]{Corollary}
\newtheorem{conjecture}[fact]{Conjecture}
\newtheorem*{conjecture*}{Conjecture}
\newenvironment{remark}{\begin{rremark} \rm}{\end{rremark}}

\newenvironment{example}{\begin{exa} \rm}{\end{exa}}

\author{\'Akos K.\ Matszangosz}
\address{Alfr\'ed R\'enyi Institute of Mathematics, Budapest, Hungary}
\email{matszangosz.akos@gmail.com}

\thanks{This research was partially supported by the Hungarian National Research, Development and Innovation Office, NKFIH K 119934.}
\title{On the cohomology rings of real flag manifolds: Schubert cycles}

\begin{document}
	\maketitle
\begin{abstract}
	We give an algorithm to compute the integer cohomology groups of any real partial flag manifold, by computing the incidence coefficients of the Schubert cells. For even flag manifolds we determine the integer cohomology groups, by proving that any torsion class has order 2 (generalizing a result of Ehresmann). We conjecture this to hold for any real flag manifold. We obtain results concerning which Schubert varieties represent integer cohomology classes, their structure constants and how to express them in terms of characteristic classes. For even flag manifolds and Grassmannians we also describe Schubert calculus. The Schubert calculus can be used to obtain lower bounds for certain real enumerative geometry problems (Schubert problems).
\end{abstract}

	\tableofcontents
\pagebreak
\section{Introduction}\label{sec:intro}

In this paper we study the integer and rational coefficient cohomology of real partial flag manifolds. 

\subsection{Additive structure}
\subsubsection{Incidence coefficients}
Complex flag manifolds have complex cell decompositions which makes the computation of CW-cohomology easy: all boundary maps are trivial. The generators of the cohomology groups can be represented by Schubert varieties whose multiplicative structure constants are given by Schubert calculus which is a classical and well-developed theory \cite{KleimanLaksov1972}.

Real flag manifolds also have a cell decomposition, however the boundary maps are no longer trivial. The boundary maps have been first examined by Ehresmann \cite{Ehresmann1937}: he computed them completely for the case of real Grassmannians $\Gr_p(\R^{p+q})=\Fl_{p,q}^\R$ \cite[p.\ 80]{Ehresmann1937}, see also \cite[p.\ 73]{Chern1951}, up to sign for flag manifolds of type $\Fl_{p,q,r}^\R$  \cite[p.\ 85]{Ehresmann1937} and he determined the cycles in the case of $\Fl_{1,q,r}^\R$ \cite[p.\ 87]{Ehresmann1937}. Ehresmann also observed that in the general case all incidence coefficients are 0 or $\pm 2$. This implies that the mod 2 cohomology groups of flag manifolds have an additive basis given by the Schubert cycles. Their mod 2 multiplicative structure constants follow from a theorem of Borel and Haefliger \cite{BorelHaefliger1961}: they agree with the structure constants of the complex Schubert cycles mod 2.

More generally, $R$-spaces (the flag manifolds of real semisimple Lie groups) have Bruhat cell decompositions \cite{DuistermaatKolkVaradarajan1983}. If all multiplicities of the restricted roots are greater than 1, then there are no cells of neighboring dimensions \cite{DuistermaatKolkVaradarajan1983}. In this case, the boundary relations of this cell decomposition are trivial, so additively the cohomology groups are freely generated by the closures of the Bruhat cells. If the multiplicities are not such (as in $\operatorname{SL}(n,\R)/P$), the boundary relations are no longer trivial, and to determine the cohomology groups, one has to compute the homology of a chain complex. %
Kocherlakota \cite{Kocherlakota1995} computed the differentials in the Morse complex for general $R$-spaces up to sign. As he remarks, the open cells determined by the Morse function coincide with the Bruhat cells, so his computations determine the incidence coefficients which are 0 or $\pm2$. 

In order to compute the integer or rational coefficient cohomology groups, the signs are also required. The latest development is the work of Rabelo and San Martin \cite{RabeloSanMartin}, who complete Kocherlakota's computation for $R$-spaces by determining the signs of the incidence coefficients via a CW homology approach. 

Our first aim is to compute the incidence coefficients of the Schubert cells in real flag manifolds via a slightly different approach, namely by using the geometry of the Schubert cells. We give an alternative proof of Kocherlakota's theorem, then we compute the signs (Theorems \ref{thm:incidencecoeffs}, \ref{thm:Kocherlakota}, \ref{thm:signs}). These results can be summarized as follows:
\begin{theorem*}
	The incidence coefficient of the Schubert cells $\Om_I$ and $\Om_J$ is given by
	$$[\Om_I,\Om_J]=\begin{cases}
		0,\qquad & N_I(a,b) \text{ even} \\
		(-1)^{s(I,J)}2,\qquad & N_I(a,b) \text{ odd} 
\end{cases}$$
where $N_I(a,b)$ and $s(I,J)$ are integers determined by the combinatorics of the ordered set partitions $I$ and $J$, see \eqref{eq:osp}, \eqref{eq:GreaterLessTangentNormal} and Theorem \ref{thm:signs}.
\end{theorem*}

We expect that the method presented here also generalizes to $R$-spaces, however in this paper we only consider the real partial flag manifolds $\Fl_{\D}^\R$. Our results are similar to \cite{RabeloSanMartin}, although the results are not directly comparable. Using Theorems \ref{thm:incidencecoeffs} and \ref{thm:signs}, we computed several examples with SageMath's homology package \cite{sagemath}. Based on these computations, we formulated a conjecture stating that all torsion in $H^*(\Fl_\D^\R;\Z)$ has order exactly 2 (see also Theorem \ref{thm:2torsion}).

\subsubsection{Cycles}Once the incidence coefficients are known, it is a nontrivial combinatorial problem to determine what the integer or rational cohomology groups are, which Schubert varieties are cycles and what the further generators are (i.e.\ which union of oriented Schubert cells). This is currently unsolved for general real flag manifolds $\Fl_\D^\R$. {The second contribution of this paper is that we determine which Schubert varieties are nonzero rational cycles in the case of Grassmannians and even flag manifolds $\Fl_{2\D}^\R$, see Theorems \ref{thm:Schubertcycles} and \ref{thm:realGrassmannianadditive}. In particular, we obtain the following result (for the notation, see \eqref{eq:osp} and Section \ref{subsec:evencycles}).}
\begin{theorem*} A basis of $H^*(\Fl_{2\D}(\R^{2N});\Q)$ is given by
$$H^*(\Fl_{2\D}(\R^{2N});\Q) =\left \bra [\si_{DI}]: I\in \binom{N}{\D}\right\ket.$$
\end{theorem*}
Using the incidence coefficients we also computed some small examples of geometric cycles generating $H^*(\Fl_\D^\R;\Q)$ for general $\D$, see the tables of Appendix \ref{sec:tables}. These tables illustrate the stark contrast of the general case with the simple descriptions of Theorems \ref{thm:Schubertcycles} and \ref{thm:realGrassmannianadditive}. 
\subsection{Structure constants}
Once the cycles have been determined, the next step is to determine the multiplicative structure constants of the cycles. We will carry this out for the cycles of $\Fl_{2\D}^\R$ and $\Gr_k(\R^n)$ with rational coefficients, using the theory of circle spaces (\cite{thesis}, \cite{FeherMatszangoszupcoming}), see Theorem \ref{thm:doubleflagcirclespace}, its Corollaries \ref{cor:realLR}, \ref{cor:realBGG} and Propositions \ref{prop:realGrassmannianeven}, \ref{prop:realGrassmannianodd}. Namely, in Corollary \ref{cor:realLR} we obtain the following result.
\begin{theorem*}
	The structure constants of $[\si_{DI}^\R]\in H^*(\Fl_{2\D}^\R;\Q)$ agree with the structure constants of $[\si_I^\C]\in H^*(\Fl_\D^\C;\Q)$:
	$$[\si_{DI}^\R]\cdot[\si_{DJ}^\R]=\sum_{K}c_{IJ}^K [\si_{DK}^\R]\acsa [\si_{I}^\C]\cdot[\si_{J}^\C]=\sum_{K}c_{IJ}^K [\si_{K}^\C],$$
\end{theorem*}

In the context of algebraic geometry, the Chow-Witt rings of real Grassmannians have been recently considered in \cite{Wendt2018}; the similarity of Propositions \ref{prop:realGrassmannianeven}, \ref{prop:realGrassmannianodd} with \cite[Theorem 1.2]{Wendt2018} suggests that Corollaries \ref{cor:realLR}, \ref{cor:realBGG} have analogues for the Chow-Witt rings of even flag manifolds.
\subsection{Characteristic classes}
Returning to the complex case, another kind of description of the cohomology ring of the complex flag manifolds is given in terms of characteristic classes of their tautological bundles. Namely, $H^*(\Fl_{\D}(\C^N);\Z)$ is generated as an algebra by the Chern classes $c_i(D_j)$ of the tautological quotient bundles $D_j=S_j/S_{j-1}$. In modern language, this can be formulated as surjectivity of the Kirwan map \cite{Kirwan1984}. The relations are given by the identity $\prod_{j=1}^m c_*(D_j)=1$, where $c_*$ is the total Chern class, see e.g.\ \cite[Chapter 23]{BottTu}. In the case of the Grassmannians, the relationship between these two descriptions is given by the \emph{Giambelli formula}
$$ [\si_{\la}]=\det(c_{\la_i+j-i}(Q)).$$

In the real case, Pontryagin classes do not always generate the cohomology ring $H^*(G/P;\Q)$; this is only the case if $G$ and $P$ have the same rank, i.e.\ even real flag manifolds $\Fl_{2\D}^\R$. In other words, the ``rational real Kirwan map" is 
surjective iff $\rk P =\rk G$. In this case, we express $[\si_\la]$ in terms of Pontryagin classes, see Corollary \ref{cor:realBGG}.

Casian and Kodama \cite{CasianKodama} made a conjecture about the ring structure of $H^*(\Fl_{\D}^\R;\Q)$ in the case of $\D=(k,n-k)$, i.e.\ Grassmannians $H^*(\Gr_k(\R^n);\Q)$, which has been proved even equivariantly via different approaches, see \cite{Takeuchi1962}, \cite{He}, \cite{Sadykov}, \cite{Carlson}. Recently, He \cite{He2019} determined the cohomology ring $H^*(\Fl_\D^\R;\Q)$ for arbitrary $\D$, for an alternative proof, see \cite{thesis}. We state He's theorem in the form convenient for us in Theorem \ref{thm:realflagcohomologyCartan}.

\subsection{Summary of the results on $H^*(\Fl_{\D}^\R;\Q)$}\label{subsec:summary}
The new results partially answer the following questions. Given a real partial flag manifold $\Fl_\D^\R$:
\begin{itemize}
	\item[Q1)] Which Schubert varieties $\si_I$ are cycles? ($I\in \binom N \D$, see \eqref{eq:osp}) Which ones are nonzero in $H^*(\Fl_\D^\R;\Q)$? Which linear combinations of Schubert cells are the remaining generators?
	\item[Q2)] What are the multiplicative structure constants of the cycles?
	\item[Q3)] What are the relations between Pontryagin classes of the tautological bundles? What further additional generators $r_i$ are there and what are the relations?
	\item[Q4)] How to express one set of generators from the other? $\si_I(p_i,r_i)=?$ $p_i(\si_I)=?$, $r_i(\si_I)=?$
\end{itemize}

Theorems \ref{thm:Schubertcycles}, \ref{thm:realGrassmannianadditive} and Appendix \ref{sec:tables} are results of type Q1). Theorem \ref{thm:doubleflagcirclespace}, Corollary \ref{cor:realLR}, and Propositions \ref{prop:realGrassmannianeven}, \ref{prop:realGrassmannianodd} concern Q2). He's Theorem  \ref{thm:realflagcohomologyCartan} \cite{He2019} answers Q3) with rational coefficients. Corollary \ref{cor:realBGG} concerns Q4).

Note, that Q1), Q3) and Q4) imply Q2) rationally, at least in theory; in practice giving combinatorial rules to compute the structure constants is not immediate and has been extensively studied in the complex case for different kind of cohomology theories by Littlewood-Richardson rules, checkers, puzzles \cite{Fulton1997}, \cite{Vakil2006}, \cite{KnutsonTao2003}. By Corollary \ref{cor:realLR}, the same combinatorial descriptions can be applied in the even real case $\Fl_{2\D}^\R$.

\subsection{Integer coefficients}
The final step is determining the integer coefficient cohomology. The formulas for the incidence coefficients (Theorems \ref{thm:incidencecoeffs}, \ref{thm:signs}) can be used in a computational homology program, to compute the cohomology groups and their generators. We calculated several examples using SageMath's homology package (all partial flag manifolds of $\R^N$, $N\leq 7$ and some up to $\R^{11}$, see Appendix \ref{sec:tables} for some of the results). Based on these computations, we make the following conjecture:
\begin{conjecture*}
	All torsion is of order exactly two in $H^*(\Fl_\D^\R;\Z)$. 
\end{conjecture*}
If the conjecture is true, then the cohomology groups can be completely determined (cf.\ Proposition \ref{prop:ranks}). The conjecture is known in the following cases. For infinite Grassmannians it is classical \cite{Borel1967} that all torsion is of order 2. For finite Grassmannians, this is a result of Ehresmann \cite{Ehresmann1937}. We prove the following result (Theorem \ref{thm:2torsion}):
\begin{theorem*}
	All torsion in $H^*(\Fl_{2\D}^\R;\Z)$ is of order exactly two.
\end{theorem*}	
In particular, the integral classes of the Schubert cycles $[Z]$ are completely determined by their rational and mod 2 reductions. The proof of this theorem involves computing the Bockstein cohomology of $X:=\Fl_{2\D}^\R$: $$H_\be^*(X):=(H^*(X;\F_2);\Sq^1),$$ ($\Sq^1\circ \Sq^1=0$); $H_\be^*(X)$ is the first page of the Bockstein Spectral Sequence. In the computations, we use that $\Sq^1[\si_I]$ equals the sum of those Schubert classes $[\si_J]$ which have nonzero incidence coefficient with $[\si_I]$ (Proposition \ref{prop:Steenrod}) this extends an observation of Lenart \cite{Lenart1998}. This concludes the determination of the cohomology groups $H^*(\Fl_{2\D};\Z)$. We do not take on the task of determining the integer coefficient ring structure.

\subsection{Applications -- real enumerative geometry}
Using the Schubert cycle description of the rational cohomology ring structure, we give an application to real enumerative geometry. Whereas in the complex case, the answer to an enumerative geometry problem is a single number, in the real case, the answer is a list of possible numbers, depending on the generic configuration. In general, very little is known about the complete range of such numbers. In the case of flag manifolds, the enumerative geometric problems are called Schubert problems and in general, the range of possible solutions is unknown. However, the cohomology ring calculation provides a lower bound. In many cases, the lower bound is 0, however in some cases, there is a meaningful lower bound, see Proposition \ref{prop:realSchubert}.

\smallskip
\textbf{Acknowledgment.} This paper is based on the author's PhD thesis \cite{thesis}. I am grateful to my PhD supervisor L\'aszl\'o M.\ Feh\'er for many valuable discussions and helpful comments that improved the quality of this paper.

\section{Preliminaries}\label{sec:preliminaries}
In this section we recall the general definition of the Vassiliev complex and introduce some notation for the geometry of real flag manifolds. 

\subsection{Vassiliev complex - incidence coefficients}\label{subsec:Vassilievincidence}
Throughout this section let $(X_\al)_{\al\in A}$ be a stratification of a smooth manifold $X$ where each stratum is contractible. Vassiliev \cite{Vassiliev1988}, \cite[Ch.\ 4.2]{AVGL} gave a method for computing the cohomology of $X$ and determining when the closure of a stratum has a fundamental cohomology class (or in another terminology is a `cycle'). {Let $F_i:=X\su \clos{X^i}$ be the open codimension filtration of $X$, where $X^i$ is the union of the $i$-codimensional strata. The \emph{Vassiliev complex} is the bottom row of the $E_1$-page of the spectral sequence associated to this filtration with connecting homomorphism 
$$ d:H^p(F_p,F_p\su X^p)\to H^{p+1}(F_{p+1},F_{p+1}\su X^{p+1}),$$
The cohomology of the Vassiliev complex computes $H^*(X;\Z)$, since the spectral sequence of the filtration degenerates on the $E_2$-page by the contractibility assumption on the strata. 
The $k$-cochains in the Vassiliev complex are the linear combinations of the $k$-codimensional strata $X_\al$, and the differential $d$ can be written as 
$$ dX_\al=\sum_{\codim X_\be=\codim X_\al +1}n_{\al\be}X_\be.$$

We will also use the notation $[X_\al,X_\be]$ instead of $n_{\al\be}$.
In case one can find a submanifold $D$ of $X$ intersecting all strata of $\clos{X_\al}$ transversally, one can compute the incidence coefficients $n_{\al\be}$ geometrically as follows (\cite{AVGL}, \cite{Vassiliev1988}, see also \cite{FeherRimanyi2002}).}
If $X_\be\not\se \clos{X_\al}$ then $n_{\al\be}=0$.

Let $A:=X_\al$ be a $k$-codimensional and $B:=X_\be\se \clos{X_\al}$ be a $k+1$-codimensional cooriented stratum, with normal bundles $\nu_\al,\nu_\be$ respectively. Let $D:=D^{k+1}$ be a $k+1$-dimensional submanifold of $X$ that intersects all strata of $\clos A$ transversally and intersects $B$ in $b$. 
Let $L:=D\cap A$ which by transversality is a disjoint union of connected curves $L_i$, whose closure contains $b$. For each $L_i$, choose a splitting $s$ of the quotient map $q:TX|_{L_i}\to \nu_\al|_{L_i}$, such that
$$ TD|_{L_i}=TL_i\oplus \nu_\al|_{L_i},$$
where $\nu_\al|_{L_i}\leq TX|_{L_i}$ via the splitting $s$. 
Then $TL_i$ is oriented by taking the orientation pointing towards $b$ and $\nu_\al$ is oriented by the coorientation of $A$, so they determine an orientation of $TD|_{L_i}$ (fix the convention of taking $TL_i$ first, then $\nu_\al|_{L_i}$). This extends to an orientation $O_1$ of $TD$ at $b$.
Since $D$ is transversal to $B$, the orientation of $\nu_\be|_b$ determines an orientation $O_2$ of $TD|_b$. If the two orientations $O_1$ and $O_2$ agree for $L_i$, then set $n_{\al\be}^i:=+1$, otherwise $-1$. Then $n_{\al\be}=\sum_i n_{\al\be}^i$.

One can show that $X_\al$ is a cycle in the Vassiliev complex, iff $Z=\clos{X_\al}$ has a fundamental cohomology class: from now on we will simply say that \emph{$Z$ is a cycle}. We will compute the Vassiliev complex of real partial flag manifolds for the stratification by Schubert cells in Section \ref{sec:realflagcohomology}.

\subsection{Geometry of flag manifolds}\label{subsec:realflaggeometry}
This section is standard, see \cite{Brion2005}, \cite{Fulton1997} for the complex case. We include it to fix some notation and properties that we will use in the computations of Section \ref{sec:realflagcohomology}. We are interested in the real case, so $\GL(N)$ denotes $\GL(N,\R)$ and $B^+$ is the subgroup of real upper triangular matrices.

\subsubsection{Schubert varieties, orbit structure}
Denote the standard basis in $\R^N$ by $e_1\stb e_N$, their one-dimensional spans by $\ep_i=\bra e_i\ket$, and the standard flag $E_j=\bigoplus_{i=1}^j\ep_i$. The stabilizer of $E_\bullet$ in $G:=\GL(N)$ is $B^+$. Choose a parabolic subgroup, i.e.\ $B^+\leq P\leq G$. Similarly to the complex case,
$$P=\GL(\D),\qquad \text{ for some } \D=(d_1\stb d_m)$$
which is the subgroup of block upper-triangular matrices with elements of $\GL(d_i)$ on the diagonal and arbitrary entries above the blocks. 

The corresponding homogeneous space $X=G/P$ is the partial flag manifold $\Fl_{\D}(\R^N)$\label{word:partialflag}. Using this notation $d_i$ denotes the difference in the dimensions of the flags, introduce $\S=(s_1\stb s_m)$, $s_i-s_{i-1}=d_i$ for their dimensions.

The $B^+$-orbits on $X$ are called \emph{Bruhat cells}\label{word:Bruhatcells} $\Om_I(E_\bullet)$. Each of these contains a unique coordinate flag $E^I_{\bullet}\in \Fl_\D(\R^N)$; they are indexed by \emph{ordered set partitions} $I\in \OSP(\D)$, where 
\begin{equation}\label{eq:osp}
\OSP(\D):=\binom{N}{\D}=S_N/(S_{d_1}\times \ldots \times S_{d_m}),
\end{equation} 
in particular $I_j\in \binom N {d_j}$, $j=1\stb m$ and $\amalg_j I_j=[N]$ (for $N\in \Z$, denote $[N]:=\{1\stb N\}$). (Caution: $E_\bullet$ denotes a complete flag, $E_\bullet^I$ a partial one.) 

\textbf{Notation.} We will denote $I\in \OSP(\D)$ by the minimal length element in $S_N$ in the coset of $I$: list elements of $I_1$ in increasing order, then elements of $I_2$ in increasing order etc.\ -- the $I_j$ separated by brackets or commas. In particular, for complete flag manifolds $\D=(1^N)$ this coincides with the one-line notation of $\OSP(\D)=S_N$, the same convention as \cite{Fulton1997}, see also \cite[p.\ 20]{FultonPragacz}. It is sometimes convenient to write $I\in \binom N \D$ as a function: $I:[N]\to [m]$ satisfying $|I^{-1}(j)|=d_j$ for all $j$. 

Given a general complete flag $A_\bullet$, the Bruhat cells coincide with the following \emph{Schubert cell}\label{word:Schubertcells} description:
\begin{equation}\label{eq:Schubertincidence} \Om_I(A_\bullet) = \{ F_\bullet\in \Fl_{\D}(\R^N): \dim F_i\cap A_{k}=r_I(i,k)\},\end{equation}
where $r_I(i,k)=\#\{l\in I_1\cup \ldots \cup I_i: l\leq k\}$.
When we omit the flag from the notation $\Om_I$, that means that we take the standard flag $E_\bullet$. For the dimension of $\Om_I$, $I\in \OSP(\D)$, introduce $\ell(I)$ be the number of \emph{inversions} (i.e.\ pairs of elements in reverse order): 
\begin{equation}\label{eq:ellI}
\ell(I):=|\{(a,b): a>b, a\in I_\al, b\in I_\be, \al<\be\}|.
\end{equation}
Then $\dim \Om_I=\ell(I)$ (see Proposition \ref{tangentspaces}).

The closure of the orbit $\Om_I$ is called a \emph{Schubert variety}\label{word:Schubertvariety} and is denoted $\si_I$. If a Schubert variety $\si_I$ is a cycle (in the sense discussed in Section \ref{subsec:Vassilievincidence}), we call it a \emph{Schubert cycle}\label{word:Schubertcycle} and its class $[\si_I]$ a \emph{Schubert class}\label{word:Schubertclass}. The orbit structure is described by the \emph{Bruhat order} (cf.\ \cite[Theorem 2.3.2]{Kocherlakota1995} for the real case):
\begin{equation}\label{orbitclosure}
\si_I=\bigcup_{J\leq I}\Om_{J}
\end{equation}
where $J\leq I$ iff $(J_{1}\cup \ldots \cup J_i)_{rtiv}\leq (I_{1}\cup \ldots \cup I_i)_{rtiv}$ for all $i$, where $rtiv$ means ``reordered to increasing value" and the partial order $(a_1\stb a_j)\leq (b_1\stb b_j)$ is the lexicographic one. This is also equivalent to $r_I(i,k)\geq r_{J}(i,k)$ for all $i,k$. 

\subsubsection{Tangent bundle of $\Fl_\D(\R^N)$}
We recall a well-known decomposition of the tangent bundle of $X$ in terms of tautological bundles. 

Let $G:=\GL(N)$ and $P\leq G$ be a parabolic subgroup; $P=\GL(\D)$. $P$ has projections to subgroups $p_i:P\to \GL(s_i)$ which are homomorphisms, whose defining representations induce the tautological bundles. For example, the defining representation of $\GL(s_i)$ on $\R^{s_i}$ induces the $i$th tautological bundle over $G/P$:
$$ S_i\iso \GL(N)\times_P \R^{s_i}.$$

The quotient and difference bundles are defined by the following exact sequences of bundles over $X$: 
\begin{equation}\label{SES}
\begin{split}
\xymatrix{			
	0\ar@{->}[r]& S_i\ar@{->}[r]^-{}&  \R^N\ar@{->}[r]&
	Q_i\ar@{->}[r]&0}\\
\xymatrix{
	0\ar@{->}[r]& S_{i-1}\ar@{->}[r]^-{}&  S_i\ar@{->}[r]&
	D_i\ar@{->}[r]&0}
\end{split}
\end{equation}
with the convention $S_0=0$. Notice that $s_i=\dim S_i$, $d_i=\dim D_i$, and let $q_i=\dim Q_i$. Recall the following general fact about the tangent bundle of homogeneous spaces:
\begin{proposition}
	Let $X=G/H$ be a homogeneous space and let $\gg$ and $\hh$ denote the Lie algebras of $G$ and $H$ respectively. Then the $G$-equivariant vector bundle $TX\to X$ fits into the short exact sequence of $G$-equivariant bundles
	$$0\to G\times_H \hh \to G\times_H\gg \to G\times_H(\gg/\hh)\iso TX\to0$$
	where $H$ acts on $\gg,\hh$ via the adjoint representation.
\end{proposition}

\begin{corollary}\label{cor:tangentbundleofflag}
	$$ TX\iso \bigoplus_{i=1}^{m-1}\Hom(D_i,Q_i)\iso\bigoplus_{1\leq i<j\leq m}\Hom(D_i,D_j)$$
\end{corollary}
\begin{proof}
	Apply the Proposition to the homogeneous space $\Fl_{\D}(\R^N)=\GL(N)/\GL(\D)$.
\end{proof}

A choice of the basis $e_i\in \R^N$ induces splittings $Q_i\to \R^N$, $D_i\to \R^N$, which is not essential, but facilitates computations, in particular it realizes $TX$ as a subbundle of $\End(\R^N)$. Using this identification and Corollary \ref{cor:tangentbundleofflag}, the tangent and normal spaces of the $B^+$-orbits can be described as follows (the proof involves computing the stabilizer subgroups $\Stab_{B^+} (E_\bullet^I)$):
\begin{proposition}\label{tangentspaces}
	The tangent and normal spaces of $\Om_I$ at $E^I_\bullet$, $I\in \OSP(\D)$ are given by
	$$T_I\Om_I=\bigoplus_{(c,d)\in T_I} \ep_{cd},\qquad N_I\Om_I=\bigoplus_{(c,d)\in N_I} \ep_{cd}$$
	where $\ep_{cd}=\Hom(\ep_c,\ep_d)$ and $$T_I:=\{(c,d)\in [N]^2: c>d,\,I(c)<I(d)\},\qquad N_I:=\{(c,d)\in [N]^2: c<d,\,I(c)<I(d)\}.$$ In particular, the dimension of $\Om_I\se \Fl_{\D}(\R^N)$ is given by $|T_I|=\ell(I)$.
\end{proposition}

\subsubsection{Direct sum maps}\label{subsec:directsum}
We are going to make use of the following natural maps between flag manifolds. Let $\D_1, \D_2 \in \N^m$ be two ordered sets of natural numbers, where we allow zero and let $\D_1+\D_2$ be their element-wise sum. The direct sum of $A_1$ and $A_2$ induces the following \emph{direct sum map} of flag manifolds:
$$ F_{\D_1,\D_2}:\Fl_{\D_1}(A_1)\times \Fl_{\D_2}(A_2)\inj \Fl_{\D_1+\D_2}(A_1\oplus A_2)$$
defined by
$$ F_{\D_1,\D_2}(F^1_\bullet,F^2_\bullet)_{\ka}:=F^1_\ka\oplus F^2_\ka.$$

This map is a $\GL(A_1)\times \GL(A_2)\leq \GL(A_1\oplus A_2)$-equivariant embedding. If $B^+\leq \GL(A_1\oplus A_2)$ is the stabilizer of a complete flag $E_\bullet\leq A_1\oplus A_2$, such that $E_\bullet=\pi_1 E_\bullet\oplus \pi_2 E_\bullet$ for $\pi_i:A_1\oplus A_2\to A_i$, then $$ B^+_i:=B^+\cap \GL(A_i)\leq \GL(A_i)$$
is a Borel subgroup and a $B^+_1\times B^+_2$-orbit embeds to a $B^+$-orbit. 

A direct sum decomposition $V=\bigoplus_i A_i$ induces a direct sum decomposition
$$\End(V)=\bigoplus_{i,j}\Hom(A_i,A_j),$$
in particular we obtain inclusions $\iota_{j}:\End(A_j)\inj \End(V)$.

\begin{proposition}\label{Endbundle}
	Let $\Fl_{\D_i}(A_i)$ be two flag manifolds with $\D_i\in \N^m$, and let $E_\bullet\in \Fl_{\D_2}(A_2)$ be a fixed flag. Then 
	$$ f:\Fl_{\D_1}(A_1)\to \Fl_{\D_1+\D_2}(A_1\oplus A_2)$$
	defined by $f:=F_{\D_1,\D_2}(\cdot,E_\bullet)$ is an isomorphism onto its image, and
	$$ df:T\Fl_{\D_1}\to T\Fl_{\D_1+\D_2}$$
	coincides with $\iota_1|_{T\Fl_{\D_1}}$, where $T\Fl_{\D_1}\leq \End(A_1)$ and $T\Fl_{\D_1+\D_2}\leq \End(A_1\oplus A_2)$, using the identification as a subbundle determined by bases $(a^1_i\in A_1)$, $(a^2_i\in A_2)$.
\end{proposition}

\section{The Vassiliev complex of $\Fl_\D^\R$} \label{sec:realflagcohomology}
In this section we compute the Vassiliev complex of $\Fl_{\D}^\R$ as described in Section \ref{subsec:Vassilievincidence}, {whose cohomology computes $H^*(\Fl_{\D}^\R;\Z)$}. The stratification is given by the Schubert cells $\Om_I$. Explicitly, we determine the incidence coefficients $[\Om_I,\Om_J]$. To compute the coefficients $[\Om_I,\Om_J]$, one has to coorient each $\Om_I$, and compare these coorientations by extending them to the adjacent orbits $\Om_J$ along transversal submanifolds $D$ as described in Section \ref{subsec:Vassilievincidence}. The {computations have} many similarities to the one given by Kocherlakota, but instead of Morse theory we emphasize the geometry of the Schubert cells. Let us give a brief outline of the proof. 

In \textbf{Section \ref{sec:cochains}} we define the coorientation of the Schubert cells $\Om_I$. In \textbf{Section \ref{sec:normalslices}} we define the transversal submanifolds $D$ to be opposite Bruhat cells. The intersection of $W:=\Om_I\cup \Om_J$ with $D$ is a \emph{Richardson curve} $R\iso \RP^1$. In \textbf{Section \ref{sec:incidence}} we show that $[\Om_I,\Om_J]$ is 0 or $\pm 2$ depending on whether the restriction of the normal bundle $\nu(W)|_R$ is trivial or not. To determine triviality of $\nu(W)|_R$ we decompose it into a direct sum of line bundles, see \textbf{Sections \ref{subsec:TXR}--\ref{sec:TW}}. This yields a combinatorial description of $[\Om_I,\Om_J]$ in terms of the number of M\"obius bundles over $R$. In \textbf{Section \ref{subsec:evencycles}} using this combinatorial description, we determine which Schubert varieties $\si_I$ are nonzero rational cycles in the even case $\Fl_{2\D}^\R$. In \textbf{Section \ref{subsec:Kocherlakota}} we relate our computations to the theorem of Kocherlakota.  By considering local orientations of $\nu(W)|_R$, we determine the sign of $[\Om_I,\Om_J]$ in \textbf{Section \ref{subsec:signs}}. To conclude the chapter we illustrate the results on $\Fl(\R^4)$. We will use the notation of Section \ref{subsec:realflaggeometry}.

We remark that in the case of real (and complex) flag manifolds, each $B^+$-orbit is homeomorphic to an affine space, so the orbit stratification yields a cell decomposition. Therefore computing the incidence coefficients agrees with the incidence coefficients of the CW complex, which have been examined by Ehresmann \cite{Ehresmann1937} for Grassmannians, later by Kocherlakota for generalized real flag manifolds ($R$-spaces) using the Morse complex \cite{Kocherlakota1995} and most recently by Rabelo and San Martin \cite{RabeloSanMartin} using CW homology (for the general case of $R$-spaces). 
\subsection{Coorientation of the strata}\label{sec:cochains}
We describe the Vassiliev complex of $X=\Fl_{\D}(\R^N)$, $\D=(d_1\stb d_m)$. First, we coorient all cells by fixing a coorientation of $\Om_I$ at $E^I_\bullet$ (see Section \ref{subsec:realflaggeometry} for the notation). 
Using the decomposition of the tangent and normal spaces given in Proposition \ref{tangentspaces}, orient both the tangent and normal spaces by the lexicographic ordering of those $e_{kl}=(e_k\mapsto e_l)$ which appear in them. 
Since $\Om_I$ is contractible, the orientation of the normal space $N_I\Om_I$  at $I$ determines a coorientation on the whole of $\Om_I$. In fact, we will not make use of the choice of orientations up until Section \ref{subsec:signs} when we determine signs of the incidence coefficients. 
\subsection{Richardson curves}\label{sec:normalslices}
Our aim is to determine the incidence numbers $[\Om_I,\Om_J]$, for $\ell(J)=\ell(I)-1$ and $J\leq I$ (recall the notations \eqref{eq:ellI} and \eqref{orbitclosure}). The Bruhat order implies that $J$ is obtained from $I$ by interchanging $a\in I_{\al}$ with some $b\in I_{\be}$, $a>b$, $\al<\be$ (this follows e.g.\ from \cite[Theorem 2.3.2]{Kocherlakota1995}). We call such $I,J$ (and $\Om_I,\Om_J$) \emph{adjacent} and fix this data in the upcoming discussion.

According to the construction of the Vassiliev complex (Section \ref{subsec:Vassilievincidence}), we will fix a transversal submanifold to $\Om_J$ at $E_\bullet^J$; natural candidates are the dual Schubert cells, i.e.\ the $B^-$-orbits. The $B^-$-orbits $B^-E^J_\bullet$ have the following characterization: 
$$ B^-E^J_\bullet=\Om_{J^D}(E^\vee_\bullet)=\{ F_\bullet\in \Fl_{\D}(\R^N): \dim F_i\cap E^\vee_{k}=r_{J^D}(i,k)\}$$
where $J^D_i:=N+1-J_i$ and $E^{\vee}_\bullet$ is the standard dual flag:
$$ E^{\vee}_i=\bra e_N\stb e_{N-i+1}\ket.$$ 

Since the flags $E_\bullet,E_\bullet^\vee$ are transverse, all $B^-$-orbits are transverse to $\Om_I$, so $B^-E_\bullet^J$ is a transversal submanifold to $\Om_J$ at $E_\bullet^J$. For general $I,J$, the intersections $\si_I^J=\si_I(E_\bullet)\cap \si_{J^D}(E_\bullet^\vee)$ are called \emph{Richardson varieties}. To determine the incidence numbers $[\Om_I,\Om_J]$, we will be interested in the \emph{Richardson curves} $\si_I^J$ when $\ell(J)=\ell(I)-1$ and $J\leq I$. Intuitively, the Richardson curve is the curve between the coordinate flags $E^I_\bullet$ and $E^J_\bullet$ obtained by continuously exchanging the coordinates $\ep_a$ and $\ep_b$ in $E_\bullet^I$.

More precisely, in terms of the direct sum maps of Proposition \ref{Endbundle}, the Richardson curve $\si_I^J$ is the isomorphic image of $f:\PP(A_1)\inj \Fl_\D(\R^N)$ for $A_1:=\ep_a\oplus \ep_b$, $A_2:=A_1^\vee=\bigoplus_{i\neq a,b}\ep_i$ and $E_\bullet:=E^{I}_\bullet\cap E^J_\bullet$, $$f(\cdot)=F_{\D_1,\D_2}(\cdot,E_\bullet),\qquad \D_2=(d_1\stb d_\al-1\stb d_\be-1\stb d_m), $$ 
for $\D=\D_1+\D_2$ and $\D_1=(1,1)$ in positions $\al,\be$. Note that $\si_I^J$ is isomorphic to $\RP^1$. The isomorphism $f$ induces tautological bundles $\rho\to \si_I^J$ on the Richardson curves as follows. 
Let $$\tilde{f}=(\id,f):A_1\times \PP(A_1)\to A_1\times \Fl_\D$$ be the trivial bundle map covering $f$ and let $\tau\to \PP(A_1)$ denote the tautological subbundle of $A_1$. Then we can define a tautological bundle over $\si_I^J$ by
\begin{equation}\label{eq:rho}
\rho:=\tilde{f}(\tau) 
\end{equation}
which is a subbundle of the trivial bundle $A_1\leq \R^N$ over $\si_I^J$.

Note that the intersection $\Om_I\cap B^-E^J_\bullet$ is the Richardson curve minus two points $\si_{I}^J\su \{E^I_\bullet,E^J_\bullet\}$, i.e.\ $\RP^1$ minus two points. 
We remark that the branches of the Richardson curve correspond to the \emph{pairs of flows} in the terminology of Kocherlakota \cite{Kocherlakota1995}.
\subsection{Incidence coefficients}\label{sec:incidence}
Let $W:=\Om_I\cup \Om_J$ and let $D:=B^-E_\bullet^J\cup B^-E_\bullet^I$ which is a transversal submanifold to $\Om_J$ at $E_\bullet^J$. Sometimes to alleviate notation we will use an abuse of notation and denote $I:=E_\bullet^I$ and $J:=E_\bullet^J$. Note that $W$ and $D$ are smooth; indeed by normality of Schubert varieties, the singularities of $\clos{\Om_I}$ are of codimension at least 2 and are unions of Schubert cells, so $\Om_I$ union the adjacent cells is always smooth. For an illustration of the notation, see Figure \ref{fig:Richardson}.

The Richardson curve $R:=\si_I^J$ is the transversal intersection of $W$ and $D$ (the intersection of the two surfaces on Figure \ref{fig:Richardson}). Let $R_+\cup R_-=R\su \{I,J\}$ denote the two branches of the Richardson curve (the choice of the sign is arbitrary). Note that $R_\pm$ are the curves denoted by $L_i$ in Section \ref{subsec:Vassilievincidence}. To compute $[\Om_I,\Om_J]$, we are going to use smoothness of $W$ and $D$.

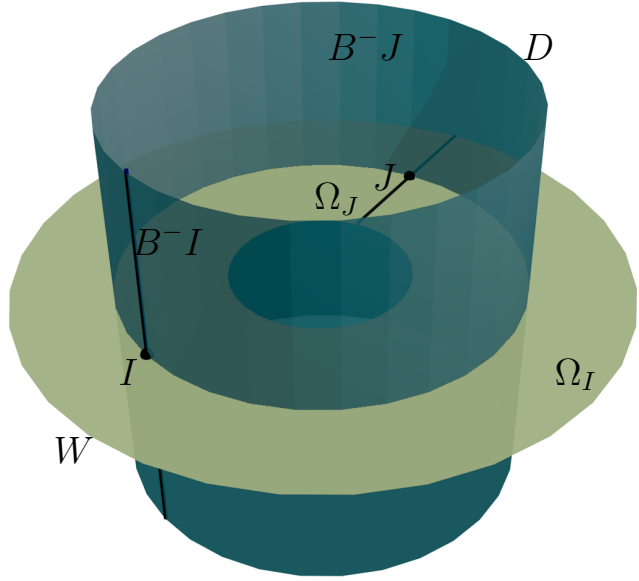
\begin{figure}
	\begin{overpic}[width=0.5\textwidth,tics=10]{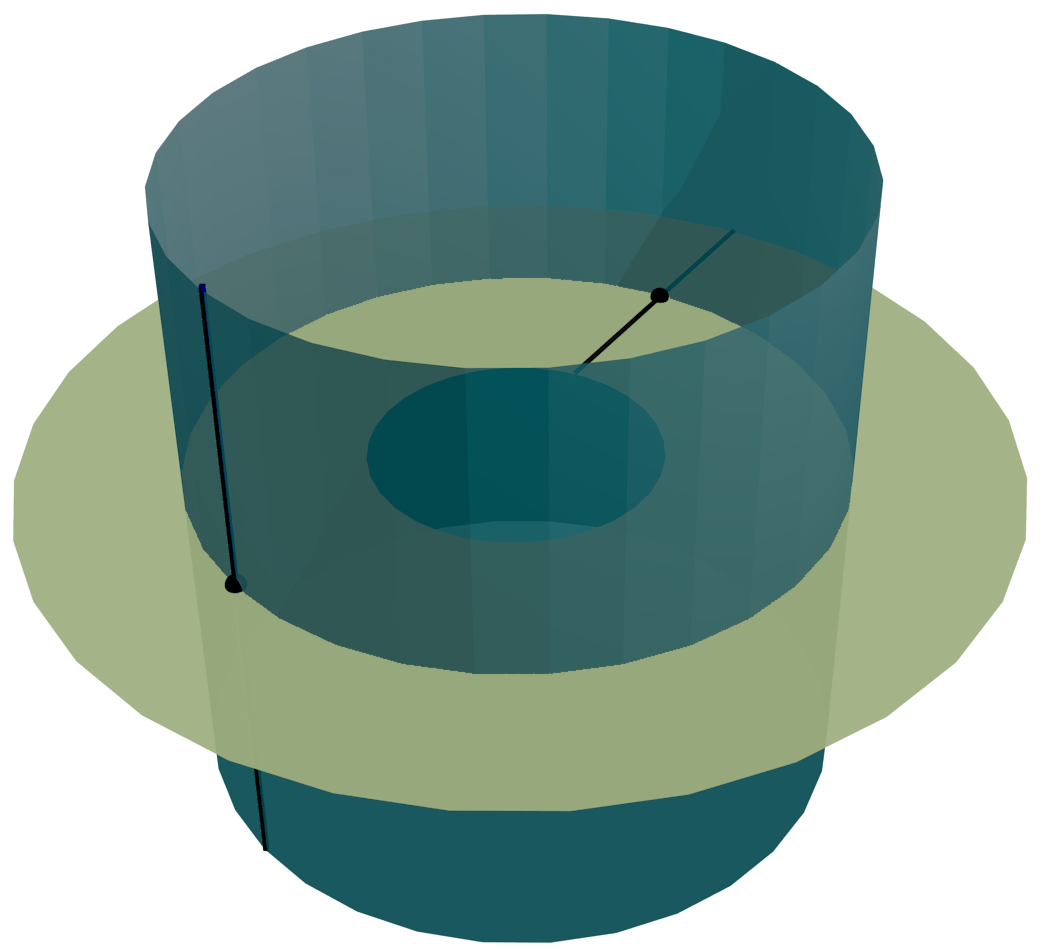}
		\put (8,18) {\large$\displaystyle W$}
		\put (80,80) {\large$\displaystyle D$}
		
		\put (18,30) {\large$\displaystyle I$}
		\put (57,60) {\large$\displaystyle J$}

		\put (50,80) {\large$\displaystyle B^-J$}		
		\put (20,50) {\large$\displaystyle B^-I$}				

		\put (85,30) {\large$\displaystyle \Om_I$}
		\put (48,57) {\large$\displaystyle \Om_J$}

	\end{overpic}	
\caption{An illustration of Richardson curves} \label{fig:Richardson}
\end{figure}

\begin{proposition}\label{incidenceasbundle}
	$$[\Om_I,\Om_J]=\begin{cases}
	0\qquad &\nu(W\inj X)|_R \text{ trivial}\\
	\pm2\qquad &\nu(W\inj X)|_R \text{ nontrivial}\\
	\end{cases}$$
\end{proposition}
\begin{proof}
	Since $R$ is the transversal intersection of $W$ and $D$, there is a short exact sequence 
	\begin{displaymath}
	\xymatrix{			
		0\ar@{->}[r]& TR\ar@{->}[r]^-{}&  TD|_R\ar@{->}[r]& \nu(W)|_R\ar@{->}[r]&0}
	\end{displaymath} 
	where $\nu(W)$ is the normal bundle of $W$ in $X$. Take a splitting of this short exact sequence:
	$$ TD|_R=TR\oplus N_W.$$
For the differentials of the Vassiliev complex, one has to compare the following two orientations for each branch $R_{\pm}$: 
	\begin{itemize}
		\item $TD|_J$ oriented by the coorientation of $\Om_J$ 
		\item $N_W|_J$ oriented by extending the coorientation of $\Om_I|_{R_\pm}$ to $J$ and $TR$ oriented towards $J$ on both branches $R_{\pm}$.
	\end{itemize}
	To compute $[\Om_I,\Om_J]$ up to sign, it is enough to compare how the two coorientations $\Om_I|_{R_\pm}$ extend to $N_W|_J$. This amounts to deciding orientability of the bundle $N_W\iso \nu(W)|_R$. If $N_W$ is orientable, then since its orientation on both branches agrees with its orientation at $I$, the orientations of $N_W|_{R_\pm}$ extend to $J$ identically. 
	Since the orientations of $TR$ induced by the orientations of $TR_\pm$ differ at $J$, in this case $[\Om_I,\Om_J]=0$. 	
	
	If $N_W$ is not orientable, then since the orientations of $N_W|_{\clos{R}_+}$ and $N_W|_{\clos{R}_-}$ agree at $I$, they are different at $J$.  In this case $[\Om_I,\Om_J]=\pm 2$. 
\end{proof}

In the upcoming Sections \ref{subsec:TXR}--\ref{sec:TW} we determine triviality of $\nu(W)|_R$ by giving linearly independent line subbundles $\la_{cd}\leq \nu(W)|_R$ spanning it (Theorem \ref{thm:decomposeTWNW}), and counting the nontrivial ones (since $R\iso S^1$, each $\la_{cd}$ is either a M\"obius bundle or a trivial one). Kocherlakota computes the incidence coefficients (up to sign) using a very similar idea: he computes the relative orientations of pairs of flows from $I$ to $J$, which are in our terminology the branches of the Richardson curves.
\begin{remark}
	As we have mentioned before, since Schubert varieties are normal, the singularities have codimension at least 2 and the singular part is a union of Schubert cells. Therefore $\Om_I$ union the adjacent cells is smooth. Let us denote this union by $\overset{\circ}{\si_I}$. This gives a new stratification of $\si_I$, with empty one codimensional stratum, but now the strata are no longer contractible. Now $\si_I$ is a cycle if and only if $\overset{\circ}{\si_I}$ is coorientable. Indeed, by the previous Proposition, this is the information encoded in $[\Om_I,\Om_J$]: the normal bundle of $\overset{\circ}{\si_I}$ restricted to the Richardson curve $\si_I^J$ is orientable iff this coefficient vanishes. Then $\overset{\circ}{\si_I}$ is coorientable iff $\nu(\overset{\circ}{\si_I})$ restricted to the Richardson curve $\si_I^J$ is orientable (trivial) for all adjacent $J$. 
	
	For general stratified submanifolds, the union with the one codimensional strata is not smooth, but when it is, this method is sufficient to decide cycleness. However to compute the cohomology groups we need more, namely to determine the incidence coefficients, which cannot be deduced only from orientability.
\end{remark}
\subsection{Splitting $TX|_R$}\label{subsec:TXR}
To determine triviality of $\nu(W)|_R$, we split $TX|_R$ into line subbundles $\la_{cd}$, parametrized by $T_I\amalg N_I$ (for the notation $T_I,N_I$, see Proposition \ref{tangentspaces}). We will show that $TW|_R=\bigoplus_{(c,d)\in T_I}\la_{cd}$, so $\bigoplus_{(c,d)\in N_I}\la_{cd}$ is isomorphic to $\nu(W)|_R$, see Theorem \ref{thm:decomposeTWNW}.

In \eqref{eq:lacd} we will specify $\la_{cd}\to R$ as subbundles of $TX|_R\leq \End(\R^N)$, in particular each $\la_{cd}$ is of the form $\Hom(\mu_1,\mu_2)$: $\mu_i\in\{\rho,\rho^\vee,\ep_k: k\in[N]\}$, where $\rho\to R$ is the tautological bundle (defined in \eqref{eq:rho}) and $\rho^\vee$ denotes its orthogonal complement.\\

Recall the quotient bundles introduced in \eqref{SES}. A choice of a basis $e_i\in \ep_i$ induces a scalar product on $\R^N$; this realizes the quotient bundles $Q_i, D_i$ as subbundles of $\R^N$. This induces the following splittings over $X=\Fl_\D(\R^N)$:
$$ \R^N=S_i\oplus Q_i=\bigoplus_{j=1}^m D_j=\bigoplus_{k=1}^N\ep_k,\qquad S_i=\bigoplus_{j=1}^i D_j,\qquad Q_i=\bigoplus_{j=i+1}^m D_j$$
for all $i=1\stb m$. By restricting to the Richardson curve $R=\si_I^J$,
\begin{equation}\label{eq:DiR}
D_i|_R=\begin{cases}
\bigoplus_{j\in I_{i}} \ep_j, \qquad &i\neq \al,\be,\\
\rho\oplus \bigoplus_{a\neq j\in I_{\al}} \ep_j, \qquad &i=\al \\
\rho^\vee \oplus \bigoplus_{b\neq j\in I_{\be}} \ep_j, \qquad &i=\be
\end{cases}
\end{equation}
where $\rho\to R$ is the tautological bundle defined previously in \eqref{eq:rho}. Then via the isomorphism
$$TX\iso \bigoplus_{i<j}\Hom(D_i,D_j)\leq \End(\R^N)$$ 
the decomposition \eqref{eq:DiR} induces a splitting of $TX|_R$ into line bundles $\la_{cd}\to R$ parametrized by  $(c,d)\in T_I\amalg N_I$, defined as follows
\begin{equation}\label{eq:lacd}
\la_{cd}:=\begin{cases}
\Hom(\rho,\rho^\vee),\qquad &\text{ if }c=a, d=b\\
\Hom(\rho,\ep_d),\qquad &\text{ if }c=a,\, I(a)<I(d)\leq I(b), d\neq b\\
\Hom(\ep_c,\rho^\vee),\qquad &\text{ if }d=b,\, I(a)\leq I(c)<I(b), c\neq a\\
\ep_{cd},\qquad &\text{ else.} \\
\end{cases}
\end{equation}

where in the else line we use that $\Hom(\rho\oplus \rho^\vee,\ep_k)=\Hom(\ep_a\oplus\ep_b,\ep_k)$. 
\subsection{A special case}\label{subsec:flag111}
In these next two sections, we show that $\{\la_{cd}:(c,d)\in T_I\}$ span $TW|_R$. We show this by reducing the general case $\Fl_{\D'}$ to the flag manifolds $\Fl_\D(\R^3)$, which is the subject of this section. The general case of $\Fl_{\D'}$ can be reduced to this computation, by using direct sum maps $f:\Fl_{\D}\inj \Fl_{\D'}$: we will show $\la_{cd}\leq TW|_R$ by showing that $\la_{cd}\leq df(TW')$ for some smooth submanifold $W'\se \Fl_{\D}$.

Let $\D=(1,1,1)$, $\ep_i=\bra e_i\ket$ and $E_\bullet$ be the standard flag, $E_i=\oplus_{j=1}^i\ep_j$. Then
\begin{itemize}
	\item $\si_{321}=X$
	\item $\si_{231}=\{F_\bullet: F_1\leq E_2\}$
	\item $\si_{312}=\{F_\bullet: E_1\leq F_2\}$
	\item $\si_{213}=\{F_\bullet: F_2= E_2\}$
	\item $\si_{132}=\{F_\bullet: F_1= E_1\}$
	\item $\si_{123}=\{F_\bullet: F_1= E_1, F_2=E_2\}$
\end{itemize}
where we use the one-line notation as discussed after \eqref{eq:osp}. All of these Schubert varieties are smooth. The tangent bundles of the $\geq 2$-dimensional Schubert varieties are therefore:
\begin{equation}
\begin{split}
 T\si_{231}=&\Hom(S_1,E_2/S_1)\oplus \Hom(D_2,D_3),\\
 T\si_{312}=&\Hom(S_1,D_2)\oplus \Hom(S_2/E_1,D_3),\\
 T\si_{321}=&TX.\\
\end{split}
\end{equation}
Let $I,J\in \OSP(\D)$, $\ell(J)=\ell(I)-1$ and $J$ be obtained by $a\in I_\al \leftrightarrow b\in I_\be$, $a>b$, $\al<\be$. If $I$ is fixed, denote the Richardson curve $\si_I^J$ by $R_{ab}$. Then by restricting to the Richardson curves $R=R_{ab}$, we obtain the expressions for $T\si_{I}|_{R}/TR$ described in Table \ref{tab:NuWFl111}, where $\ep_{ij}=\Hom_\R(\ep_i,\ep_j)\leq TX|_R$ and $\rho\to R$ is the tautological bundle as described earlier. 
\begin{table}
\begin{center}
	\begin{tabular}{ | l || c | c | c |}
		\hline
		$(a,b)$ & (2,1) & (3,1) & (3,2) \\ \hline
		$\si_{321}$ & $\ep_{32}\oplus \ep_{31}$  & $\ep_{32}\oplus \ep_{21}$ & $\ep_{31}\oplus \ep_{21}$\\ \hline
		$\si_{231}$ & $\Hom(\ep_3,\rho^\vee)$ & $\ep_{21}$ & --\\ \hline
		$\si_{312}$ & -- & $\ep_{32}$ & $\Hom(\rho,\ep_1)$\\
		\hline
	\end{tabular}
\end{center}
\caption{The bundle $T\si_I|_{R}/TR$, for $R=R_{ab}=\si_I^J$}
\label{tab:NuWFl111}
\end{table}
Table \ref{tab:NuWFl111} shows that $\{\la_{cd}:(c,d)\in T_I\}$ defined in \eqref{eq:lacd} spans $TW|_R$.

Similarly, for $\D=(1,2)$, $\Fl_{\D}(\R^3)=\PP^2$, the only $\geq 2$-dimensional orbit is $I=(3)(1,2)$, $J=(2)(1,3)$ and 
$$T\si_I|_{R}/TR=\Hom(\rho, \ep_1)$$
In case $\D=(2,1)$, $\Fl_{\D}(\R^3)=\Gr_2(\R^3)$, the only $\geq 2$-dimensional orbit is $I=(2,3)(1)$, $J=(1,3)(2)$ and 
$$T\si_I|_{R}/TR=\Hom(\ep_3, \rho^\vee)$$
\subsection{Decomposing $TW$}\label{sec:TW}
Let us return to the general case $X=\Fl_\D^\R$, and fix adjacent $I,J$, $a\in I_\al, b\in I_\be$, $R=\si_I^J$, $W=\Om_I\cup \Om_J$ as before.

In this section we show that $TW|_R=\bigoplus_{(c,d)\in T_I}\la_{cd}$ for $\la_{cd}$ defined in \eqref{eq:lacd}. We show this by embedding smooth submanifolds $f:W'\inj W$, such that $\la_{cd}\leq df(TW')|_R\leq TW|_R$ for all $(c,d)\in T_I$. The $W'$ are submanifolds of smaller flag manifolds $\Fl_{\D_1}$ which are embedded in $\Fl_\D(\R^N)$ via the direct sum maps of Section \ref{subsec:directsum}.

For each $(c,d)\in T_I$ {distinct from $(a,b)$} we specify a direct sum map. Set $\vartheta:=\{a,b,c,d\}$ {which has $3$ or $4$ elements}, and let
$$(\D_1)_\ka:=|\{k\in\vartheta:I(k)=\ka\}|, \qquad \ka=1\stb m$$
the number of distinct elements in $\vartheta$ which are in $I_\ka$ (this is either 0, 1 or 2). Let $\D_2:=\D-\D_1$ and $\Theta=\bra \ep_k:k\in \vartheta\ket$. The decomposition $\R^N=\Theta\oplus \Theta^\vee$ induces the direct sum map
$$ F:\Fl_{\D_1}(\Theta)\times \Fl_{\D_2}(\Theta^\vee) \to \Fl_{\D}(\R^N).$$

Let $E_\bullet^{IJ}:=E_\bullet^I\cap E_\bullet^J\in \Fl_{\D_2}(\Theta^\vee)$. Define the embedding $f:\Fl_{\D_1}(\Theta)\inj \Fl_\D(\R^N)$ by $F(\cdot,E_\bullet^{IJ})$.
\begin{proposition}\label{prop:RWflag}
	Given $I,J$ as above, there exist (unique) $I', J'\in \binom{|\vartheta|}{\D_1}$, such that the following diagram commutes:

$$	\xymatrix{			
		R'	\ar@{^{(}->}[r]\ar@{->}[d]^{\iso}& W'\ar@{^{(}->}[r]\ar@{^{(}->}[d]& \Fl_{\D_1}(\Theta)\ar@{^{(}->}[d]^f\\
		R	\ar@{^{(}->}[r]& 	W\ar@{^{(}->}[r]&\Fl_{\D}(\R^N)& 
	}$$

	where $R'=\si_{I'}^{J'}\se \Fl_{\D_1}(\Theta)$, $W'=\Om_{I'}\cup \Om_{J'}\se \Fl_{\D_1}(\Theta)$ are smooth submanifolds.
\end{proposition}
\begin{proof}
	
	Since $f$ is an embedding, $E^I_\bullet$ and $E^J_\bullet$ have at most one preimage each. There are unique order preserving maps 
	$$n:\{1\stb |\vartheta|\}\to \vartheta,\qquad p:\{1\stb |I(\vartheta)|\}\to I(\vartheta).$$ 
Let the maps
	$$I',J':\{1\stb |\vartheta|\}\to \{1\stb |I(\vartheta)|\}$$
	be defined by $I'(i):=p^{-1}(I(n(i)))$, $J'(i):=p^{-1}(J(n(i)))$. Then $f(E^{I'}_\bullet)= E^I_\bullet$ and $f(E^{J'}_\bullet)=E^J_\bullet$. 
	Since $f$ is $\GL(\Theta)$-equivariant, $f(R')=R$, and $f(W')\se W$.
\end{proof}
\begin{corollary}
	$df(TW')\leq TW$.
\end{corollary}
We will use the following theorem to determine triviality of $\nu(W)|_R$.
\begin{theorem}\label{thm:decomposeTWNW}
	$TW|_R$ is an inner direct sum of the subbundles $\{\la_{cd}:(c,d)\in T_I\}$ where $\la_{cd}$ are defined in \eqref{eq:lacd}. 	$\nu(W)|_R$ is isomorphic to $\bigoplus_{(c,d)\in N_I}\la_{cd}$.
\end{theorem}
We split the proof below into Proposition \ref{lacd} and Corollaries \ref{lacdTW}, \ref{lacdNU}.

\begin{proposition}\label{lacd}
	$\{\la_{cd}:(c,d)\in T_I\}$ are subbundles of $TW|_R$, where $\la_{cd}$ are defined in \eqref{eq:lacd}.
\end{proposition}
\begin{proof}

	Let $(c,d)\in T_I$ and set $\vartheta=\{a,b,c,d\}$ as above, $\vartheta_1=\{a,b\}$ and $\vartheta_2=\vartheta\su \vartheta_1$ (this has either 1 or 2 elements).
	
	If $|\vartheta|=4$, then by Proposition \ref{tangentspaces} $(c,d)\in T_J$. In this case we can further decompose $\D_1$ as $$(\D_{1i})_\ka=|\{k\in\vartheta_i:I(k)=\ka\}|,\qquad \ka=1\stb m$$
	$\D_1=\D_{11}+\D_{12}$ and $\Theta=\Theta_1\oplus \Theta_2$ for $\Theta_i=\bra \ep_k:k\in \vartheta_i\ket$. Note that $\D_{1i}=(1,1)$, ($i=1,2$) in the appropriate positions. This decomposition induces another direct sum map	
	$$g:\PP \Theta_1\times \PP \Theta_2\to \Fl_{\D_1}(\Theta).$$
	
	Let $W'':=\PP\Theta_1\times(\PP\Theta_2\su\PP\ep_d)$ and $R'':=\PP\Theta_1\times \PP\ep_c$.
	As in the proof of Proposition \ref{prop:RWflag}, equivariance shows that $W':=\Om_{I'}\cup \Om_{J'}=g(W'')$ and $R'=g(R'')$. Then $\ep_{cd}\leq TW''|_{R''}$. Applying Proposition \ref{Endbundle} for $g$ and $f$, 
	$$ \la_{cd}=d(f\circ g)\ep_{cd}\leq TW|_R.$$
	
	If $|\vartheta|=3$, we are in the case of $\D_1=(1,1,1)$, $\D_1=(2,1)$ or $\D_1=(1,2)$. Then Table \ref{tab:NuWFl111} shows that the Proposition holds for $I',J'\in \binom{3}{\D_1}$ and $(c',d')\in T_{I'}$. By Proposition \ref{Endbundle}, $\la_{cd}\leq TW|_R$.
\end{proof}
\begin{corollary}\label{lacdTW}
	$\{\la_{cd}: (c,d)\in T_I\}$ are linearly independent and therefore span $TW|_R$.
\end{corollary}
\begin{proof}
	Set $\vartheta_1=\{a,b\}$ and $\Theta_1=\ep_a\oplus \ep_b$. It is enough to show that the bundles $\la_{cd}$ are linearly independent in each summand
	$$ \End(\R^N)=\End(\Theta_1^\vee)\oplus \End(\Theta_1)\oplus \bigoplus_{k\neq a,b} (\Hom(\Theta_1,\ep_k)\oplus \Hom(\ep_k,\Theta_1))$$
	Given $c,d$, set $\vartheta=\{a,b,c,d\}$, $\vartheta_2=\vartheta\su\vartheta_1$,  and $\Theta_2=\bra\ep_k:k\in \vartheta_2\ket$. 
	
	If $|\vartheta|=4$, then $\la_{cd}=\ep_{cd}$ which are linearly independent in $\End(\Theta_1^\vee)$. 
	
	If $|\vartheta|=3$, then $|\vartheta_2|=1$, denote its single element by $k$. Then 
	$$\la_{cd}\leq \Hom(\Theta_1,\ep_k)\oplus \Hom(\ep_k,\Theta_1)$$
	For fixed $k$ there are at most 2 such $(c,d)$ pairs, since $c>d$ for $(c,d)\in T_I$. So it is enough to check linear independence of such pairs. 
	
	If $a>k>b$, then $\la_{ak}\leq \Hom(\Theta_1,\ep_k)$, $\la_{kb}\leq \Hom(\ep_k,\Theta_1)$, so they are independent.

	If $a>b>k$ then in order for $(a,k),(b,k)\in T_I$ to hold, $I(a)<I(b)<I(k)$ must hold, then $\la_{ak}=\ep_{ak},\la_{bk}=\ep_{bk}$ are independent in $\Hom(\Theta_1,\ep_k)$. The case $k>a>b$ is similar.

	Finally, $\la_{ab}=\Hom(\rho,\rho^\vee)\leq \End(\Theta_1)$.
\end{proof}
\begin{corollary}\label{lacdNU}
	$\nu(W)|_R$ is isomorphic to $\bigoplus_{(c,d)\in N_I}\la_{cd}$.
\end{corollary}
This concludes the proof of Theorem \ref{thm:decomposeTWNW}.

Let us introduce some notation. Fix $I\in \OSP(\D)$, $\D=(d_1\stb d_m)$, and fix $a,b\in \{1\stb N\}$ such that $\al= I(a)$, $\be=I(b)$, $\al<\be$. For $c\in\{1\stb N\}$, $\ga,\de\in \{0,1\stb m\}$, set
\begin{equation}\label{eq:GreaterLessTangentNormal}
\begin{split}
G_I(c,\ga,\de):=|\{d>c: \ga<I(d)\leq \de\}|,\qquad &L_I(c,\ga,\de):=|\{d<c:\ga< I(d)\leq \de\}|\\
T_I(a,b):=L_I(a,\al,\be)+G_I(b,\al-1,\be-1),\qquad& N_I(a,b):=G_I(a,\al,\be)+L_I(b,\al-1,\be-1)\\
G_I(a,\de):=G_I(a,\de,m),\qquad &G_I(a):=G_I(a,\al),
\end{split}
\end{equation}
abbreviating Greater, Less, Tangent, Normal. For the fixed adjacent $J\leq I\in \OSP(\D)$, $J$ is obtained by interchanging $a\in I_{\al}$ with some $b\in I_{\be}$, $a>b$, $\al<\be$. 

Note that by Definition \eqref{eq:lacd} of $\la_{cd}$, $N_I(a,b)$ is the number of nontrivial $\la_{cd}$ for $(c,d)\in N_I$. Then by Proposition \ref{incidenceasbundle} and Theorem \ref{thm:decomposeTWNW}, we have:
\begin{theorem}\label{thm:incidencecoeffs}
	If $I,J\in \OSP(\D)$, $\ell(J)=\ell(I)-1$ and $J$ is obtained from $I$ by interchanging $a\in I_\al \leftrightarrow b\in I_\be$, $a>b$, $\al<\be$, then using the notations \eqref{eq:GreaterLessTangentNormal}:
	\begin{equation}\label{eq:incidence}
	[\Om_I,\Om_J]=\begin{cases}
	0,\qquad & N_I(a,b) \text{ even} \\
	\pm2,\qquad & N_I(a,b) \text{ odd} 
	\end{cases} \end{equation}
\end{theorem}

\begin{example}
	Let us compute the incidence coefficient $ [\Om_I,\Om_J]$ up to sign for 
	$$I=[3\,6,1\,4,2\,5] ,\qquad J=[2\,6,1\,4,3\,5]$$ 
in $\Fl_{2,2,2}$ (recall the notation introduced after \eqref{eq:osp}). In this case $a=3$, $\al=1$, $b=2$, $\be=3$. There are two elements in $I_2\cup I_3$ (4 and 5) greater than $a=3$,
so $G_I(a,\al,\be)=2$. There is one element in $I_1\cup I_2$ which is less than $b=2$, so $L_I(a,\al,\be)=1$, thus $N_I(a,b)=3$ and $ [\Om_I,\Om_J]=\pm 2$.
\end{example}
To any cochain complex, one can define an \emph{incidence graph} as was done in \cite{CasianStanton1999}: the vertices of the graph are elements of $\binom{N}{\D}$ and $(I,J)$ is an edge iff $[\Om_I,\Om_J]\neq 0$. See Figure \ref{fig:incidencegraph} for the incidence graph of $\Fl(\R^4)$ at the end of this chapter (the colors correspond to the signs).

\subsection{Determining the cycles}\label{subsec:evencycles}
In the even case $\Fl_{2\D}^\R$, equation \eqref{eq:incidence} is actually sufficient to determine the rational coefficient cohomology $H^*(\Fl_{2\D}^\R;\Q)$ additively in terms of Schubert cycles, i.e.\ the sign of $\pm 2$ can be ignored. For the $\Z$-coefficient cohomology, the signs are required as well, see Section \ref{subsec:signs}.

If $2\D=(2d_1,2d_2\stb 2d_r)$ and $I\in \OSP(\D)$, then \emph{the doubled ordered set partition} $DI\in \OSP(2\D)$ is obtained by replacing each $i\in I_j$ by $(2i-1,2i)\in DI_{j}$; each element $k\in DI_j$ has a unique \emph{pair} $k'\in DI_j$. A \emph{double Schubert variety}\label{word:doubleSchubert} $\si_{DI}^\R\se \Fl_{2\D}^\R$ is a Schubert variety corresponding to $DI\in \OSP(2\D)$. In the case of the Grassmannian $\D=(k,l)$, $DI\in \binom{2(k+l)}{2k}$ corresponds to the Young diagram obtained by subdividing each square into $2\times 2$ squares in the Young diagram corresponding to $I\in \binom{k+l}{k}$, see Figure \ref{fig:doubleYoung}. In terms of ordered set partitions, for $\D=(1,1)$, the doubled ordered set partitions are $[1\,2,3\,4]$ and $[3\,4, 1\,2]\in \OSP(2\D)$, which are the doubles of $[1,2]$ and $[2,1] \in\OSP(\D)$ respectively.

\begin{figure}
	\centering
	\begin{picture}(50,30)
	\put(0,30){\line(0,-1){30}}
	\put(10,30){\line(0,-1){30}}
	\put(20,30){\line(0,-1){20}}
	\put(30,30){\line(0,-1){10}}	
	\put(40,30){\line(0,-1){10}}	
	\put(0,30){\line(1,0){40}}
	\put(0,20){\line(1,0){40}}
	\put(0,10){\line(1,0){20}}
	\put(0,0){\line(1,0){10}}
	\put(50,15){\vector(1,0){10}}
	\end{picture}$\quad$
	\begin{picture}(50,30)
	\linethickness{0.1mm}
	\put(5,30){\line(0,-1){30}}
	\put(15,30){\line(0,-1){20}}
	\put(25,30){\line(0,-1){10}}	
	\put(35,30){\line(0,-1){10}}		
	\put(0,25){\line(1,0){40}}			
	\put(0,15){\line(1,0){20}}				
	\put(0,5){\line(1,0){10}}					
	\linethickness{0.4mm}
	\put(0,30){\line(0,-1){30}}
	\put(10,30){\line(0,-1){30}}
	\put(20,30){\line(0,-1){20}}
	\put(30,30){\line(0,-1){10}}	
	\put(40,30){\line(0,-1){10}}	
	\put(0,30){\line(1,0){40}}
	\put(0,20){\line(1,0){40}}
	\put(0,10){\line(1,0){20}}
	\put(0,0){\line(1,0){10}}

	\end{picture}
	\caption{The double of a Young diagram}
	\label{fig:doubleYoung}
\end{figure}
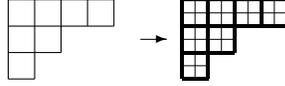

\begin{theorem}\label{thm:Schubertcycles}
	In $\Fl_{2\D}(\R^N)$ the double Schubert varieties $\si_{DI}$ are (integer) cycles and their classes $[\si_{DI}]$ generate a free $\Z$-submodule of $H^*(\Fl_{2\D}^\R;\Z)$. Rationally, $[\si_{DI}]$ form a basis of $H^*(\Fl_{2\D}^\R;\Q)$.
\end{theorem}
\begin{proof}
	Let $I=DI'\in \OSP(2\D)$ be a doubled ordered set partition. As above, if $\Om_J\se \si_I$ and $\ell(I)-\ell(J)=1$, then $J$ is obtained by $a\in I_{\al}\leftrightarrow b\in I_{\be}$, $\al<\be$, $a>b$. Since $I$ is doubled, both terms in the sum
	\begin{equation}\label{eq:NIab} N_I(a,b)=G_I(a,\al,\be)+L_I(b,\al-1,\be-1)\end{equation}
	are even; e.g.\ if $k>a$ and $I(k)>\al$, then its pair $k'$ also satisfies $k'>a$ and $I(k')=I(k)>\al$. So all coefficients $[\Om_I,\Om_J]$ vanish and $\si_I$ is a cycle.
	
	Now assume that $\Om_I\se \si_J$ and $\ell(J)-\ell(I)=1$, and $I=DI'$ be obtained by $a\in J_{\al}\leftrightarrow b\in J_{\be}$, $\al<\be$, $a>b$. We again have to determine the parity of \eqref{eq:NIab}, but now for $N_J(a,b)$. Let $a'$ and $b'$ denote the pairs of $a$ and $b$ respectively. Since $I$ is a doubled ordered set partition, $a'\in I_\be$ and $b'\in I_\al$. $\ell(J)-\ell(I)=1$ implies that $a<a'$ and $b'<b$. As before, everything in $J$ appears in pairs, except $a<a'$ and $b<b'$ which shows that $G_J(a,\al,\be)$ and $L_J(a,\al-1,\be-1)$ are both odd. So $N_J(a,b)$ is even and $\Om_I$ appears in all incidence relations with zero coefficient $[\Om_J,\Om_I]=0$. Therefore $[\si_I]$ does not appear in any relation and the double Schubert cycles $\{[\si_{DJ}]:J\in \OSP(\D)\}$ are linearly independent.

	Finally, $\dim_\Q H^*(\Fl_{2\D}^\R;\Q)=|\OSP(\D)|$, which follows e.g.\ from Theorem \ref{thm:realflagcohomologyCartan}. This agrees with the number of doubled ordered set partitions of $2\D$.
\end{proof}
\subsection{Kocherlakota's theorem}\label{subsec:Kocherlakota}
Theorem \ref{thm:incidencecoeffs} gives an alternate proof of Kocherlakota's theorem \cite[Theorem A]{Kocherlakota1995} for the special case of the classical real flag manifolds $\Fl_\D^\R$ (type A). Before stating it we have to introduce some further notation. 

Let $\gg$ be a real split semisimple Lie algebra. Let $\aa$ be a maximal $\R$-diagonalizable subalgebra and let $\Si\se \aa^*$ be the restricted root system. Since $\gg$ is split, all root multiplicities are one. Choose a regular element $\xi\in \aa$, which determines a positive Weyl chamber $C^+$ and $\Si=\Si^+\coprod \Si^-$. The reflections $r_{\varphi}$ in the root planes $\ker \varphi$, $\varphi\in \Si^+$ generate the Weyl group $W$ of the root system. The Weyl group acts freely and transitively on the Weyl chambers, and the Weyl chambers $C_w$ are labeled by $w\in W$, $C^+=C_1$ for $1\in W$. Given $H\in \clos{C^+}$, let $\Theta\se \Si^s$ be the simple roots vanishing at $H$. Then the Weyl orbit of $H$ is $W/W_H$, which parametrizes the Bruhat cells of $G/P_\Theta$ (cf.\ \cite{DuistermaatKolkVaradarajan1983}). Now we state Kocherlakota's theorem. Given $x\in \aa$, let
$$ \mathcal{N}(x):=\{\varphi\in \Si^+:\varphi(x)<0\},$$
$$\si(x):=\sum_{\varphi\in N(x)}\varphi\in \aa^*,$$ 
and $\ell(x):=|\mathcal{N}(x)|$ {(we change the notation of Kocherlakota to $\mathcal{N}$ in order to distinguish from $N_I$)}. This is consistent with notation \eqref{eq:ellI}, as we will show below, and in general it is the dimension of $\Om_x\se G/P_\Theta$ for $x\in W.H$. We will give another interpretation of $N(x)$, see \eqref{eq:NI}. Let us now recall the theorem of Kocherlakota.
\begin{theorem}[Kocherlakota]\label{thm:Kocherlakota}
	Let $x,y\in W.H=W/W_H$ and $\ell(y)=\ell(x)-1$. If $r_\varphi(x)=y$ for a reflection $r_\varphi$, $\varphi\in\Si^+$, then $\si(x)-\si(y)=m\varphi$ for some $m\in \Z$. The incidence coefficients are given by
	$$	
	[\Om_x,\Om_y]=\begin{cases}
	0,\qquad & m \text{ odd} \\
	\pm2,\qquad & m \text{ even} 
	\end{cases} $$
\end{theorem}
Before giving the proof for $\gg=\sl(N,\R)$, let us recall some specifics about the root system of type $A_{N-1}$. The roots in an appropriate basis are $\pm e_{ij}$ where $e_{ij}=e_i-e_j$, $i<j$. The simple roots are $\de_i=e_{i,i+1}$, and in terms of the simple roots $e_{ij}=\sum_{k=i}^j \de_i$. Its Weyl group is $W\iso S_N$ and the reflections $r_{ij}$ through the hyperplane $\ker e_{ij}$ correspond to the transpositions $(ij)\in S_N$.

The Weyl-orbit of a regular element $H\in C^+$ can be parametrized by $W\iso S_N$. 
If $H\in \clos{C^+}$ is not regular, list the simple roots $\de_i$ not vanishing on $H$: $\de_{s_1},\de_{s_2}\stb \de_{s_r}$, such that $s_1<s_2<\ldots s_r$, and set $s_{r+1}:=N$. Then the Weyl-orbit $W.H=W/W_H=\OSP(\D)$, where $\D=(d_1\stb d_r)$, for $d_i=s_{i+1}-s_i$.

This implies that positive roots $e_{ij}\in \Si^+$ have the following property: given $I\in W/W_H$, $e_{ij}(I)<0$ iff $(i,j)$ is an inversion of $I\in W/W_H$. Thus $\mathcal{N}(I)$ is the set of inversions of $I\in W/W_H$:
\begin{equation}\label{eq:NI}
\mathcal{N}(I)=\{e_{ij}:(i,j)\text{ is an inversion of } I\}
\end{equation}
In particular, for $I\in \OSP(\D)$, $|\mathcal{N}(I)|=\ell(I)=\dim_\R\Om_I$ as we have stated above (e.g.\ by Proposition \ref{tangentspaces}).

\begin{proof}[Proof of Theorem \ref{thm:Kocherlakota} for $\gg=\sl(N,\R)$]
	Let $I,J\in W/W_H=\OSP(\D)$, such that $r_{ab}(I)=J$, and $\ell(J)=\ell(I)-1$, $a>b$, $a\in I_\al, b\in I_\be$. By \eqref{eq:NI}, the set theoretic difference of $\mathcal{N}(J)\su \mathcal{N}(I)$ consists of those $e_{ij}$, for which $i,j$ is an inversion in $J$, but not in $I$. 
	
	Clearly all such $i,j$ pairs must contain $a$ or $b$. A simple verification shows that there are three types of elements in $\mathcal{N}(I)\su \mathcal{N}(J)$ (the other cases can be excluded using $\ell(J)=\ell(I)-1$):
	\begin{itemize}
		\item If $e_{bj}\in \mathcal{N}(I)\su \mathcal{N}(J)$ and $a<j$, then $e_{aj}\in \mathcal{N}(J)\su \mathcal{N}(I)$, 
		\item if $e_{ja}\in \mathcal{N}(I)\su \mathcal{N}(J)$ and $j<b$, then $e_{jb}\in \mathcal{N}(J)\su \mathcal{N}(I)$, and
		\item $e_{ba}\in \mathcal{N}(I)\su \mathcal{N}(J)$. 
	\end{itemize}
	Since $e_{bj}-e_{aj}=e_{ba}$ for $b<a<j$ and $e_{ja}-e_{jb}=e_{ba}$ for $j<b<a$
	$$ \si(I)-\si(J)=\sum_{e_{ij}\in \mathcal{N}(I)\su \mathcal{N}(J)}e_{ij} -\sum_{e_{ij}\in \mathcal{N}(J)\su \mathcal{N}(I)}e_{ij}=(G_I(a,\al,\be)+L_I(b,\al-1,\be-1)+1)e_{ba}=(N_I(a,b)+1)e_{ba}$$
	using the definitions preceding Theorem \ref{thm:incidencecoeffs}.  We can conclude by Theorem \ref{thm:incidencecoeffs}.
\end{proof}
\subsection{Signs}\label{subsec:signs}
For cooriented $\Om_I$ and $\Om_J$, determining the actual signs of $[\Om_I, \Om_J]$ requires some further work. We conclude this chapter by determining the signs. We obtain similar results as \cite{RabeloSanMartin} who deal with the general case of $R$-spaces. 

One can make several choices of orientations - we coorient all $\Om_I$ lexicographically as described in Section \ref{sec:cochains} and compute the signs of $[\Om_I,\Om_J]$ relative to these orientations. The signs obtained in Theorem \ref{thm:signs} can be implemented in a computer program, and can be used to compute the cohomology groups of real flag manifolds with integer coefficients. We were mainly interested in the Schubert cycle generators of rational coefficient cohomology. We used SageMath's homology package \cite{sagemath}. See Section \ref{sec:tables} for results in some cases not covered by Theorem \ref{thm:Schubertcycles}.

\subsubsection{Geometry}
Before stating the Proposition, let us introduce some notation. Let $I, J\in \OSP(\D)$ be adjacent, obtained by $a\in I_\al \leftrightarrow b\in I_\be$. We will denote by $(e_c\mapsto e_d)\in \Hom(\ep_c,\ep_d)$ the homomorphism mapping $e_c$ to $e_d$. Let $R_+$ be a branch of the Richardson curve $R$ and let $U$ be a slight enlargement of $R_+$: a (contractible) connected open set $U\subsetneq R$ containing $\clos{R_+}$. Let $r\in \Ga(\rho|_U)$ and $r^\vee\in \Ga(\rho^\vee|_U)$ be nowhere vanishing sections, such that $r(I)=-r^\vee(J)=e_a$ and $r^\vee(I)=r(J)=e_b$ (this choice determines the branch $R_+$). Define sections of $\la_{cd}|_U$, $(c,d)\in N_I$ as follows:
$$s_{cd}:=\begin{cases}
(e_c\mapsto e_d),\qquad &\text{ if } (c,d)\in N_J\\
(r\mapsto e_d),\qquad &\text{ if }c=a,\, I(a)<I(d)\leq I(b)\\
(e_c\mapsto r^\vee),\qquad &\text{ if }d=b,\, I(a)\leq I(c)<I(b)\\
\end{cases}	$$
We will say that $s_{cd}$ is \emph{trivial} if $s_{cd}=(e_c\mapsto e_d)$ and \emph{nontrivial} otherwise. Notice that $s_{cd}$ is trivial iff $(c,d)\in N_I$ and $(c,d)\in N_J$. 
\begin{proposition}
	The sign of $[\Om_I,\Om_J]$ is +1 iff the following two orientations agree, -1 otherwise:
$$N_J^1:=\big( (e_c\mapsto e_d): (c,d)\in N_J\big ),\qquad N_J^2:=\big ( -(e_b\mapsto e_a),s_{cd}(J):(c,d)\in N_I\big )$$
where in both $N_J^1$ and $N_J^2$, the terms involving $(c,d)$ are listed lexicographically.
\end{proposition}
\begin{proof}
Recall that the incidence coefficient $[\Om_I,\Om_J]$ in the Vassiliev complex can be computed as a sum of $\pm1$ contributions for each branch $R_\pm$ of the Richardson curve (cf.\ Section \ref{subsec:Vassilievincidence}). The contribution of one of the branches $R_+$ can be computed as follows. Take the splitting 
\begin{equation}\label{eq:splitting}
	TB^-J|_{R_+}=TR\oplus \nu_I|_{R_+}, 
\end{equation}
where $\nu_I$ is the sum of line bundles of Theorem \ref{thm:decomposeTWNW}. Then the contribution of $R_+$ is obtained by comparing the following two orientations of $T_JB^-_J$:

\begin{itemize}
	\item the orientation induced by the coorientation of $\nu(\Om_J)|_J$: this is the lexicographical orientation $N_J^1=\big( (e_c\mapsto e_d): (c,d)\in N_J\big )$ and
	\item the orientation $O_2$ determined by the splitting \eqref{eq:splitting}: $TR|_{R_+}$ is oriented towards $J$, and $\nu_I|_{R_+}$ is oriented by the coorientation of $\Om_I$.
\end{itemize}
Note that the second orientation has to be extended to $J$. In order to extend the orientation of $TB^-J|_{R_+}$ to $J$, we will use the sections $s_{cd}$: if the orientation $(s_{cd}(I))$ agrees with the coorientation $\Om_I$, then the orientation $(TR|_{U},s_{cd}(J))$ determines the orientation $O_2$. On the branch $R_+$, $(r\mapsto r^\vee)$ points towards $J$, so $O_2$ is exactly $N_J^2$.
\end{proof}
Therefore the combinatorial task is to determine the (relative) sign of two signed permutations. 
\subsubsection{Combinatorics}
Before giving a notation heavy answer, let us illustrate on a simple example the computation of the signs:
\begin{example}
	Let $\D=(1^6)$, $I=(4, 5, 6, 1, 2, 3)$ and $J=(4, 2, 6, 1, 5, 3)$, so $a=5$, $b=2$, $\al=2$, $\be=5$. The normal spaces in lexicographical ordering are spanned by
	$$ N_I=(12)(13)(23)(45)(46)(56),\qquad N_J=(13)(15)(23)(25)(26)(45)(46)$$
	where $(ij)$ denotes $\pm(e_i\mapsto e_j)$ (recall the description of the normal spaces in Proposition \ref{tangentspaces}). One has to compare two orientations of $N_J\Om_J$, the first orientation being the lexicographical orientation $N_J^1$.
	The second orientation is given by $N_J^2=\big ( (ba),N_I(a\leftrightarrow b)\big )$, where $a\leftrightarrow b$ is the operation of exchanging $a$ and $b$ in $(cd)$  ($\{c,d\}\cap \{a,b\}\neq \emptyset$) if $s_{cd}$ is a nontrivial section (i.e.\ if $(cd)$ appears only in $N_I$). These orientations here are (notice that $(2,3)$ appears in both $N_I$ and $N_J$, so $s_{23}$ is trivial and there is no substitution $2\leftrightarrow 5$ for $(23)$):
	$$ N_J^1=(13)(15)(23)(25)(26)(45)(46),\qquad N_J^2=(25)(15)(13)(23)(45)(46)(26).$$
	The first difference $c_1$ comes from listing $(ab)=(25)$ first in $N_J^2$, this contributes $c_1=3$ transpositions: it precedes $(13)(15)(23)$. The second difference $c_2$ comes from exchanging $(a,d)$ with $(b,d)$ if $(a,d)\not\in N_J$; this contributes $c_2=2$ as $(26)$ succeeds $(45)(46)$.
	The third difference $c_3$ comes from exchanging all $(c,b)$ with $(c,a)$ if $(c,b)\not\in N_J$; this contributes $c_3=1$ transpositions; $(15)$ precedes $(13)$. Finally, the nontrivial sections involving $r^\vee(J)=-e_5$ obtain a sign: this contributes $c_4=2$ sign changes, $-(e_1\mapsto e_5)\in \Hom(e_1,\rho^\vee)|_J$ and $-(e_2\mapsto e_5)$. 
	So the incidence coefficient is $[I,J]=(-1)^82=+2$, where $c_1+c_2+c_3+c_4=8$. 
\end{example}

The following Theorem formalizes this computation:
\begin{theorem}\label{thm:signs}
	If $[\Om_I,\Om_J]\neq 0$, then its sign is given by $[\Om_I,\Om_J]=(-1)^{s(I,J)}\cdot 2$, where
	$$ s(I,J)=c_1+c_2+c_3+c_4$$
	and 
	$$ c_1=G_I(b,\al)-G_I(a)+\sum_{\substack{c<b}}G_I(c)$$
	$$ c_2=\left(G_I(a)-G_I(a,\be)\right)
	\left(\sum_{b<c<a}G_I(c)+G_I(a,\be)\right)$$
	$$ c_3 = \sum_{\substack{c< b\\ \al\leq I(c)< \be}}G_I(b,I(c))-G_I(a,I(c))$$
	$$ c_4 = L_I(b,\al-1,\be-1)+1 $$
\end{theorem}
\begin{proof}
	The first permutation is the elements of $N_J^1$ listed lexicographically:
	$$ \ldots,(b-1,d^{b-1}_{n_{b-1}}), (b,d^b_1)\stb(b,a)\stb (b,d^b_{n_b}),\stb (c,d_1^c)\stb (c,d_{n_c}^c)\stb (a,d_1^a)\stb (a,d^a_{n_a}),\ldots$$
	The second permutation is obtained by listing $(b,a)$ and then the elements of $N_I$ lexicographically
	$$(b,a), \ldots,(b-1,f^{b-1}_{m_{b-1}}), (b,f^b_1)\stb (b,f^b_{m_b}),\stb (c,f_1^c)\stb (c,f_{m_c}^c)\stb (a,f_1^a)\stb (a,f^a_{m_a}),\ldots$$
	and making the following substitutions, which compared to $N_J^1$ contribute a certain number of transpositions that we will determine below:
	\begin{itemize}
		\item listing $(b,a)$ first: this contributes $c_1$ many transpositions,
		\item replacing $(a,d)$ with $(b,d)$ for all nontrivial sections $s_{ad}$: $c_2$ many transpositions,
		\item replacing $(c,b)$ with $(c,a)$ for all nontrivial sections $s_{cd}$: $c_3$ many transpositions,
	\end{itemize}
	One must also count the sign differences coming from the values of $s_{cd}(J)=\pm(e_c\mapsto e_d)$. These sign differences are represented by the term $c_4$. Now we determine $c_1,c_2,c_3,c_4$.
	
	The first difference is that $(e_b\mapsto e_a)$ is the first element in $N^2_J$. This contributes
	$$c_1=|\{b<c<a:J(b)<J(c)\}|+ \sum_{\substack{c<b}}G_I(c)=G_I(b,\al)-G_I(a)+\sum_{\substack{c<b}}G_I(c)$$
	many transpositions.
	
	Given a nontrivial section $s_{ad}$, $(a,d)\in N_I$: $s_{ad}(J)=\pm(e_b\mapsto e_d)$, 
	it has the following distance from its final position at $(b,d)$ in $N_J^1$: $(a>b)$
	$$\#\{a<c<d:I(b)<I(c)\}+\#\{d<c:I(b)<I(c)\} +\sum_{b<c<a}G_I(c)=$$	$$=G_I(a,\be)-G_I(d,\be)+G_I(d,\be)+ \sum_{b<c<a}G_I(c) $$
	as $(e_a\mapsto e_d)$ swaps place with every $(e_f\mapsto e_g)$ pair whose position doesn't change and precedes it. The sum of these for nontrivial $s_{ad}$ pairs is the term $c_2$.
	
	Similarly, for a nontrivial section $s_{cb}$, $(c,b)\in N_I$, $s_{cb}(J)=\pm(e_c\mapsto e_a)$ has the following distance from its final position:
	$$\#\{b<d<a:I(c)<I(d)\}=G_I(b,I(c))-G_I(a,I(c))$$
	the sum of which for nontrivial $s_{cb}$ pairs is $c_3$.
	
	Finally, the sections $s_{cd}$ induce $L_I(b,\al-1,\be-1)+1$ many sign changes: the trivial bundles have trivial sections, the nontrivial sections $s_{cd}$ involving only $r$ introduce no sign change, whereas each bundle involving $r^\vee(J)=-e_a$ contributes a sign change, the number of which is $c_4=L_I(b,\al-1,\be-1)+1$. 
\end{proof}

\begin{remark}
	{The sign of $[\Om_I,\Om_J]$ determines a coloring of the incidence graph (see Figure \ref{fig:incidencegraph}). Taking the opposite orientation of a vertex $I$, changes the color of all edges incident to $I$, however since in the end such a graph computes the cohomology of $\Fl_{\D}$, the cohomology of the chain complex is the same.}
\end{remark}

{For another choice of orientations, see \cite{RabeloSanMartin}. There a reduced decomposition $w=r_1\cdot \ldots \cdot r_d$ is fixed for each $w\in S_N$, and such a reduced decomposition determines an ordering of the inversions of $w$, i.e.\ an orientation of $T_I$. Such an ordering is convenient, since for adjacent $w'<w=r_1\cdot \ldots \cdot r_d$, $w'$ has a reduced decomposition obtained by omitting some uniquely defined $r_i$ - then the relative orientation of $r_i,r_1\stb  \hat{r_i}\stb r_d$ and $r_1\stb r_d$ is simply $(-1)^i$. However, the initially fixed reduced decomposition of $w'$ might differ from $r_1\cdot \ldots \cdot \hat{r_i}\cdot \ldots \cdot r_d$, so one also has a term comparing these two orientations. }

\subsection{An example: $\Fl(\R^4)$}\label{subsec:Fl4}
Figure \ref{fig:incidencegraph} contains the incidence graph of $\Fl(\R^4)$ defined as follows. The vertices of the graph are the Schubert cells $X_\al$, $\al\in S_4$ and two vertices $X_\al, X_\be$ are connected by an edge if $[X_\al,X_\be]=\pm 2$. This diagram can also be found in \cite[p.\ 529]{CasianStanton1999}. The extra information is the coloring of the graph representing the signs: a blue edge represents $[\Om_I,\Om_J]=+2$ a red edge corresponds to $-2$. In this case, in order to compute the cohomology, the signs are actually not needed, as can be seen from the form of the graph; one can read off the cohomology groups of $\operatorname{Fl}(\R^4)$:
$$H^0=\Z,\quad H^1=0,\quad H^2=\Z_2^{\oplus 3},\quad H^3=\Z^{\oplus 2}\oplus \Z_2^{\oplus 2},\quad H^4=\Z_2^{\oplus 2},\quad H^5=\Z_2^{\oplus 3},\quad H^6=\Z$$
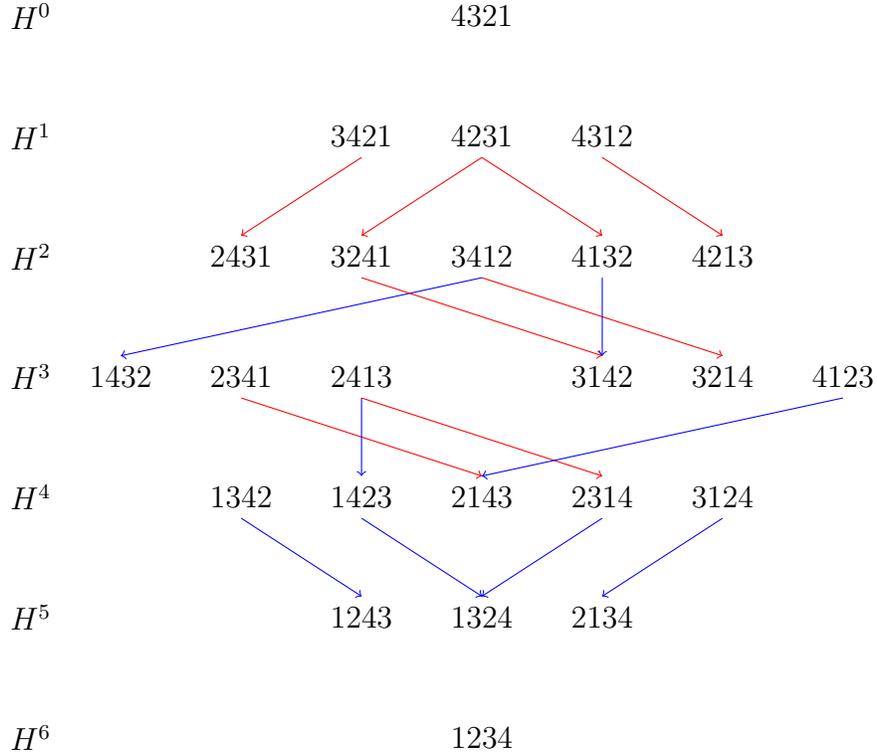
\begin{figure}
	
	\begin{center}
		\begin{tikzpicture}[
		node distance=1.6cm,
		roundnode/.style={ellipse, draw=green!60, fill=green!5, very thick},
		squarednode/.style={rectangle, draw=red!60, fill=red!5, very thick},
		diamondnode/.style={diamond, aspect=1.4, draw=red!60, fill=red!5, very thick,inner sep=0.25ex},
		]
		\node[draw=none,fill=none]      							(4321)           {4321};
		
		\node[draw=none,fill=none, below of=4321]      	(4231)           {4231};
		\node[draw=none,fill=none, left of=4231]      	(3421)           {3421};
		\node[draw=none,fill=none, right of=4231]      	(4312)           {4312};
		
		\node[draw=none, fill=none, below of=4231]      (3412)           {3412};
		\node[draw=none,fill=none, left of=3412]      			(3241)           {3241};
		\node[draw=none,fill=none, left of=3241]      			(2431)           {2431};
		\node[draw=none,fill=none, right of=3412]      			(4132)           {4132};
		\node[draw=none,fill=none, right of=4132]      			(4213)           {4213};
		
		\node[below of=3412]      						(empty1)         {};
		\node[draw=none, fill=none, left of=empty1]     (2413)           {2413};
		\node[draw=none,fill=none, right of=empty1]      		(3142)           {3142};
		\node[draw=none,fill=none, right of=3142]      			(3214)           {3214};
		\node[draw=none,fill=none, right of=3214]      			(4123)           {4123};
		\node[draw=none,fill=none, left of=2413]      			(2341)           {2341};
		\node[draw=none,fill=none, left of=2341]      			(1432)           {1432};
		
		\node[draw=none,fill=none, below of=empty1]      		(2143)           {2143};
		\node[draw=none,fill=none, left of=2143]      			(1423)           {1423};
		\node[draw=none, fill=none, left of=1423]      	(1342)           {1342};
		\node[draw=none,fill=none, right of=2143]      			(2314)           {2314};
		\node[draw=none, fill=none, right of=2314]      (3124)           {3124};
		
		\node[draw=none,fill=none, below of=2143]      			(1324)           {1324};
		\node[draw=none,fill=none, right of=1324]      			(2134)           {2134};
		\node[draw=none,fill=none, left of=1324]      			(1243)           {1243};
		
		\node[draw=none,fill=none, below of=1324]      			(1234)           {1234};
		
		\node[draw=none, fill=none] at (-6,0)      			(H0)       {$H^0$};
		\node[draw=none, fill=none, below of=H0] 					(H1)       {$H^1$};
		\node[draw=none, fill=none, below of=H1] 					(H2)       {$H^2$};
		\node[draw=none, fill=none, below of=H2] 					(H3)       {$H^3$};
		\node[draw=none, fill=none, below of=H3] 					(H4)       {$H^4$};
		\node[draw=none, fill=none, below of=H4] 					(H5)       {$H^5$};
		\node[draw=none, fill=none, below of=H5] 					(H6)       {$H^6$};
		
		
		\draw[->,red] (3421.south) -- (2431.north);
		\draw[->,red] (4231.south) -- (3241.north);
		\draw[->,red] (4231.south) -- (4132.north);
		\draw[->, red] (4312.south) -- (4213.north);
		
		\draw[->, red] (3412.south) -- (3214.north);
		\draw[->, blue] (3412.south) -- (1432.north);
		\draw[->,red] (3241.south) -- (3142.north);
		\draw[->,blue] (4132.south) -- (3142.north);
		
		\draw[->, red] (2341.south) -- (2143.north);
		\draw[->, blue] (2413.south) -- (1423.north);
		\draw[->, red] (2413.south) -- (2314.north);
		\draw[->, blue] (4123.south) -- (2143.north);
		
		\draw[->, blue] (1342.south) -- (1243.north);
		\draw[->, blue] (1423.south) -- (1324.north);
		\draw[->, blue] (2314.south) -- (1324.north);
		\draw[->, blue] (3124.south) -- (2134.north);
		\end{tikzpicture}
	\end{center}
	\caption{The signed incidence graph of $\operatorname{Fl}_4$: blue edges signify $+2$, red edges $-2$}
	\label{fig:incidencegraph}
\end{figure}
The generators of the cohomology groups are not necessarily unique: $x=1432$ and $y=3214$ generate a submodule isomorphic to $\Z\oplus \Z_2$ -- their sum is a 2-torsion element, whereas either of them generates a free $\Z$-submodule. The smallest instance where we encountered a possible dependence on the choice of the signs is $\Fl(\R^6)$: by changing the signs so that the incidence graph still defines a chain complex, the cohomology groups can be different. Therefore the signs are indeed required in certain cases.

\section{The ring structure of $H^*(\Fl_{2\D}^\R;\Q)$: Schubert calculus}\label{sec:ringFlD}

The double Schubert varieties $\si_{DI}^\R$ are cycles and their classes form a basis of $H^*(\Fl_{2\D}^\R;\Q)$ by Theorem \ref{thm:Schubertcycles}. By showing that $\Fl_{2\D}^\R$ are \emph{circle spaces} \cite{thesis}, \cite{FeherMatszangoszupcoming} we can deduce their structure constants by relating them to the structure constants of the complex case. {Thus, any formula involving Schubert cycles $[\si_I^\C]$ in a complex partial flag manifold $\Fl_\D^\C$, holds for the double real partial flag $\Fl_{2\D}^\R$ and doubled real Schubert cycles $[\si_{DI}^\R]$.} Let us first state some properties of circle spaces and state {a generalization of the Borel-Haefliger theorem} \cite{BorelHaefliger1961}.

Circle spaces are analogues of conjugation spaces introduced by Hausmann, Holm and Puppe \cite{HausmannHolmPuppe2005}. The $\Z_2$-actions are replaced by $\U(1)$-actions, and $\F_2$-coefficient cohomology is replaced by $\Q$-cohomology. Let $\Ga:=\U(1)$. If a $\Ga$-space $X$ is a circle space, then it has the following properties (which are proved analogously to \cite{HausmannHolmPuppe2005}, see \cite{thesis}). 
\begin{itemize}
	\item $X$ has nonzero cohomology in degrees $4i$,
	\item there exists a degree-halving ring isomorphism $\ka:H^{2*}(X)\to H^*(X^\Ga)$,
	\item there exists a Leray-Hirsch section (also known as cohomology extension of the fiber) $\si:H^*(X)\to H_\Ga^*(X)$ which is multiplicative, and satisfies the \emph{restriction equation}: for any $x\in H^{4d}(X)$:
	$$ r(\si(x))=\ka(x)u^d+\eta$$
	where $r:H_\Ga^*(X)\to H^*_\Ga(X^\Ga)\iso H^*(X^\Ga)[u]$, and $\eta$ is a $u$-polynomial of degree less than $d$.
	\item a pair $(\ka,\si)$ satisfying the restriction equation is unique, and they satisfy a naturality property with respect to equivariant maps between circle spaces.
\end{itemize}

Before stating the generalized Borel-Haefliger theorem we have to introduce a technical definition. By a \emph{good $\U(1)$-invariant cycle} $Z\se X$ we mean that 
\begin{itemize}
	\item $Z$ is a $\U(1)$-invariant stratified submanifold which is a cycle,
	\item its top stratum $Z_k$ is $\Ga$-invariant,
	\item $Z^\Ga$ has a stratification with unique, connected top stratum $Z_k^\Ga$ of some codimension $2l$.
\end{itemize}
These technical conditions ensure that $Z^\Ga$ is also a cycle, and that
$$ [Z\se X]_\Ga|_{X^\Ga}=w\cdot [Z^\Ga\se X^\Ga]+\eta$$
in $H^*_\Ga(X^\Ga)\iso H^*(X^\Ga)[u]$, where $w=w_0\cdot u^l$, $w_0\in \Z$ is the weight of a representation and $\eta$ is a sum of $u$-monomials of degree less than $l$. This can be shown via the Excess Intersection Formula. If $Z\se X$ is a good $\U(1)$-invariant cycle satisfying $\codim Z\se X=2\codim Z^\Ga\se X^\Ga$, then we say in short that $Z$ is a \emph{halving cycle}.
The proof of the following theorem can be found in \cite{thesis}, and an upcoming paper \cite{FeherMatszangoszupcoming}:
\begin{theorem}[Generalized Borel-Haefliger theorem]
	Let $\Ga=\U(1)$ and let $X$ be a compact oriented $\Ga$-manifold, whose rational cohomology groups have a basis of halving cycles $[Z_i]\in H^{4k_i}(X)$. Assume that the $\U(1)$-equivariant normal bundle $\nu(X^\Ga\inj X)$ has only one weight $\la\in \Z$. Then $X$ is a circle space with $\ka[Z_i]=\la^{k_i}[Z_i^\Ga]$ and $\si[Z_i]=[Z_i]_\Ga$.
	
	In particular, the assignment sending $[Z_i]$ to $[Z_i^\Ga]$ determines a degree-halving multiplicative isomorphism between $H^{2*}(X;\Q)$ and $H^*(X^\Ga;\Q)$. 
	
\end{theorem}

More generally, one can drop the assumption on orientability and only assume that $\Q$-Poincar\'e duality \cite{AlldayPuppe2007} is satisfied with the formal dimensions satisfying $\fd(X)=2\fd(X^\Ga)$. {($X$ is a $\Q$-Poincar\'e duality space if $H^{\operatorname{top}}(X;\Q)\iso \Q$ and the pairing $H^k(X)\otimes H^{\operatorname{top}-k}(X)\to H^{\operatorname{top}}(X)$ is perfect; $\fd(X):=\operatorname{top}$.)}

Our main examples of circle spaces are the real even flag manifolds $\Fl_{2\D}^\R$. 

Let us introduce a $\U(1)$-action on real flag manifolds in $\R^{2n}$: The identification of $\R^{2n}\leftrightarrow \C^n$ as real $\Ga$-representations induces an action on $\Fl_{\mathcal{E}}(\R^{2n})$, $\mathcal{E}=(e_1\stb e_r)$. 

\begin{theorem}\label{thm:doubleflagcirclespace}
	Let $\Ga:=\U(1)$. With the $\Ga$-action introduced above, $\Fl_{2\D}(\R^{2n})$ is a circle space, with $\Ga$-fixed point set $\Fl_\D(\C^n)$. Furthermore $$\ka[\si_{DI}^\R]=2^{|I|}[\si_I^\C],$$ where $[\si_I^\C]\in H^{2|I|}(\Fl_\D(\C^N))$.
\end{theorem}
\begin{proof}
	A representation theoretic computation involving the tangent bundles shows that the normal weights are all 2. So by the generalized Borel-Haefliger theorem, it is enough to show that a) the Schubert cycles $[\si_{DI}]$ form a basis of rational cohomology, b) for an appropriate complete real flag $F_\bullet$, $Z=\si_{DI}^\R(F_\bullet)$ are good $\U(1)$-invariant cycles satisfying $\codim Z=2\codim Z^\Ga$, and c) have $\U(1)$-fixed point set $Z^\Ga=\si_I^\C(F_\bullet^\C)$. The $[\si_{DI}^\R]$ form a basis by Theorem \ref{thm:Schubertcycles}, so a) holds. It remains to choose a flag $F_\bullet$ satisfying b) and c). 
	
	b) Let $F_\bullet$ be a complete flag in $\R^{2n}$, such that $F_{2i}$ are $\Ga$-invariant, and let $F_\bullet^\C$ denote the corresponding complex flag $(F_0,F_2\stb F_{2n})$ in $\C^n$ by the identification $\R^{2n}\leftrightarrow \C^n$. Then, $\si_{DI}^\R(F_\bullet)$ are halving cycles. Indeed, by the rank conditions \eqref{eq:Schubertincidence}, the complex points of $\si_{DI}^\R(F_\bullet)$ are the points of $\si_{I}^\C(F_\bullet^\C)$ since $$\dim_\R (W\cap W')=2k\iff \dim_\C (W\cap W')=k$$ for any $\Ga$-invariant subspaces $W,W'\leq \R^{2n}$. Therefore for this choice of $F_\bullet$, the $\si_{DI}^\R(F_\bullet)$ are $\Ga$-invariant and $(\si_{DI}^\R(F_\bullet))^\Ga=\si_I^\C(F_\bullet^\C)$, so c) holds. A dimension count shows $$\codim_\R\si_{DI}^\R(F_\bullet)=2\codim_\R \si_I^\C(F_\bullet^\C)$$
	and since $\Om_{I}^\C$ is the unique top stratum of $\si_{I}^\C$, the $\si_{DI}^\R$ are good $\U(1)$-invariant cycles.
	
\end{proof}

{The Corollaries below follow from the previous Theorem, multiplicativity of $\ka$ and 
	$$\ka p_j(S^\R_i)=2^jc_j(S^\C_i).$$ For further details we refer to \cite{thesis}.}
\begin{corollary}[Littlewood-Richardson coefficients]\label{cor:realLR}
	In $H^*(\Fl_{2\D};\Q)$ the structure constants are given by
	$$[\si_{DI}^\R]\cdot[\si_{DJ}^\R]=\sum_{K}c_{IJ}^K [\si_{DK}^\R]$$
	where $c_{IJ}^K$ are the Littlewood-Richardson coefficients of the complex Schubert varieties:
	$$[\si_{I}^\C]\cdot[\si_{J}^\C]=\sum_{K}c_{IJ}^K [\si_{K}^\C].$$
\end{corollary}
\begin{corollary}[Giambelli formula type description]
	\label{cor:realBGG}
	In $H^*(\Fl_{2\D};\Q)$ the Schubert cycles can be expressed in terms of characteristic classes as follows:
	$$[\si_{DI}^\R]=q(p_*(S_i^\R))\qquad \iff \qquad [\si_{I}^\C]=q(c_*(S_i^\C)),$$
	that is the same polynomial describes the double real Schubert classes and complex Schubert classes in terms of Pontryagin and Chern classes.
\end{corollary}
\begin{corollary}\label{cor:realrelations}
	$$ H^*(\Fl_{2\D}^\R)=\Q[p_*(S_i^\R)]/\mathcal{R}(p_*(S_i^\R))\acsa H^*(\Fl_\D^\C)=\Q[c_*(S_i^\C)]/\mathcal{R}(c_*(S_i^\C)),$$
	where $\mathcal{R}(x_*^i)$ denotes an ideal in the variables $x_j^i$, that is the same polynomial relations hold in the two cohomology rings in terms of Pontryagin and Chern classes of the respective tautological bundles.
\end{corollary}

\section{The Casian-Kodama conjecture: Grassmannians}\label{sec:CasianKodama}

{The Casian-Kodama conjecture \cite{CasianKodama} concerns the cohomology ring structure of real Grassmannians and was stated as follows.
For $(K,N)=(2k,2n),(2k+1,2n+1)$ or $(2k,2n+1)$:
$$ H^*(\Gr_{K}(\R^{N});\Q)\iso \Q[p_1\stb p_k,p_1'\stb p_{n-k}]/(p_*\cdot p_*')$$
where $p_i=p_i(S_1)$, $p_i'=p_i(Q_1)$, $p_*=\sum p_i$ and
$$ H^*(\Gr_{2k+1}(\R^{2n});\Q)\iso \Q[p_1\stb p_k,p_1'\stb p_{n-k},r]/(p_*\cdot p_*',r^2)$$
where $r$ is the Schubert class corresponding to the $L$-shape Young diagram: $(n-k,1^{k-1})$ (see also below).}

This characteristic class description was completely settled even equivariantly by \cite{Takeuchi1962}, \cite{He}, \cite{Sadykov}, \cite{Carlson} (see Q3) in Section \ref{subsec:summary}). However, it appears that the question of the cohomology ring structure in terms of Schubert cycles (the fundamental classes of Schubert varieties) has not yet been addressed. We give such a description by using the generalized Borel-Haefliger theorem, and via the previous computation of the incidence coefficients.

First, we describe the additive structure, see Theorem \ref{thm:realGrassmannianadditive}. To deduce the multiplicative structure, we show that $\Gr_{K}(\R^N)$ are circle spaces, except when $K$ is odd, $N$ is even, see Proposition \ref{prop:realGrassmannianeven}. For $K$ odd, $N$ even, we use an additional geometric argument to deduce the structure constants, see Proposition \ref{prop:realGrassmannianodd}.
\subsection{Additive structure}
A convenient way to parametrize the Schubert varieties in Grassmannians is by Young diagrams $\la\se K\times (N-K)$. One has the following conversion formulas between $\la\se K\times (N-K)$ and $I\in\binom{N}{K}$:
\begin{equation}\label{eq:conversion}
\la_j=N-K+j-I_j,\qquad I_j=N-K+j-\la_j.
\end{equation} 

Before stating the Schubert cycle description of $H^*(\Gr_K(\R^N);\Q)$ it is convenient to introduce the \emph{$L$-operation} on Young diagrams: given $\la\se K\times (N-K)$, let $L\la\se (K+1)\times (N-K+1)$ be the partition
$$ L\la:=(N-K+1,\la_1+1,\la_2+1\stb \la_{K}+1).$$
In terms of Young diagrams, the diagram contains the first row and column, and the complement of this L-shape is the Young diagram $\la$, see Figure \ref{fig:LYoung} (the added L-shape is marked with bullet points). We call the corresponding Schubert varieties $\si_{L\la}$ \emph{L-Schubert varieties}. Recall that $D\la\se 2k\times 2(n-k)$ denotes the double of a Young diagram $\la\se k\times (n-k)$ defined earlier (Figure \ref{fig:doubleYoung}).

\begin{figure}
	\centering
	$\yng(3,1) {\qquad \then\qquad} {\tiny \young(\bullet\bullet\bullet\bullet\bullet,\bullet\hfil\hfil\hfil,\bullet\hfil,\bullet)}$
	\caption{The $L$-operation on $\la=(3,1,0)\se 3\times 4$}
	\label{fig:LYoung}
\end{figure}

\begin{theorem}\label{thm:realGrassmannianadditive}
	$$H^*(\Gr_K(\R^N);\Q)=\begin{cases}
	\Q\Big\bra[\si_{D\la}],[\si_{L(D\la)}]:\la\se k\times (n-k)\Big\ket \qquad & \text{$N$ even $K$ odd}\\
	\Q\Big\bra[\si_{D\la}]:\la\se k\times (n-k)\Big\ket \qquad & \text{else.}	
	\end{cases} $$
	where $k=\lfloor K/2\rfloor$, $n=\lfloor N/2\rfloor$. 
\end{theorem}
\begin{proof}
	The incidence coefficients modulo sign were given in equation \eqref{eq:incidence}. For Grassmannians, this equation can be rewritten in terms of Young diagrams as follows (see also \cite{CasianKodama}). If $\mu$ is a partition obtained from $\la$ by increasing $\la_j$ by 1, then
	\begin{equation}\label{eq:Grassmannincidence}
		[\si_\la,\si_\mu]=\begin{cases}
	0\qquad& \la_j-j \text{ odd}\\
	\pm 2\qquad& \la_j-j \text{ even}
	\end{cases}
	\end{equation}
	In particular, it follows that $\si_\la$ is a cycle if and only if for all partitions $\mu$ obtained by increasing $\la_j$ by 1, $\la_j-j$ is odd. Pictorially this means that the 'inner' corners of the Young diagram $\la$ only lie on even antidiagonals. This implies that double Schubert varieties $\si_{D\la}$ and double L-Schubert varieties $\si_{L(D\la)}$ are cycles -- possibly torsion. 
	
	However, the boundary relations imply that the coefficient of $\si_{D\la}$ in every incidence relation vanishes:
	$$[\si_{D\la},\si_\mu]=0$$ 
	for all $\dim\si_\mu=\dim\si_{D\la}+1$, so the $[\si_{D\la}]$ {do not appear in any relations and therefore are linearly independent}. A similar computation shows that if $K$ odd and $N$ even, $[\si_{L(D\la)}]$ appear with zero coefficient in every incidence relation, so these classes are linearly independent.
	
	The Cartan description of Appendix \ref{subsec:Cartanmodel} implies that
	$$\dim_\Q H^*(\Gr_K(\R^N);\Q) = \begin{cases}
	2\binom n k	\qquad & \text{$N$ even $K$ odd}\\
	\binom n k \qquad & \text{else.}	
	\end{cases}$$
	which implies that these Schubert classes form a basis.
\end{proof}

\subsection{Multiplicative structure}
The $\Gr_{K}(\R^N)$ are circle spaces unless $K$ is odd and $N$ is even. If $K$ and $N$ is even, this is contained in Theorem \ref{thm:doubleflagcirclespace}. {In particular, the Giambelli and Pieri formulas hold by replacing all $[\si_\la^\C]$ with $[\si_{D\la}^\R]$ in the formulas.} The remaining cases: $K$ odd $N$ even and $K$ even $N$ odd are both nonorientable. We now proceed to show that they are also circle spaces.

Identify $\R^{2n+1}=\R\oplus \C^n$ as $\Ga=\U(1)$-representations; let the trivial representation be the first coordinate $\R=\bra e_1\ket$. This induces actions on $\Gr_{2k+1}(\R^{2n+1})$ and $\Gr_{2k}(\R^{2n+1})$, whose fixed point set can be identified with $\Gr_k(\C^n)$. 
\begin{proposition}\label{prop:realGrassmannianeven}
	{The natural inclusions induce isomorphisms with rational coefficients}
	$$ H^*(\Gr_{2k}(\R^{2n}))\iso H^*(\Gr_{2k}(\R^{2n+1}))\iso H^*(\Gr_{2k+1}(\R^{2n+1}))$$
	compatibly with the descriptions given in Corollaries \ref{cor:realLR}, \ref{cor:realBGG},  \ref{cor:realrelations}.
\end{proposition}
\begin{proof}
	Let $F_\bullet^\R\in \Fl(\R^{2n+1})$ be the standard complete flag. Then by our definition of the $\U(1)$-action, $F_\bullet^\C=(\bra e_2,e_3\ket, \bra e_2,e_3,e_4,e_5\ket\stb \bra e_2\stb e_{2n+1}\ket)$ is a complete complex flag $F_\bullet^\C\in \Fl(\C^n)$. The rank description \eqref{eq:Schubertincidence} implies that double Schubert varieties $\si_{D\la}(F_\bullet^\R)$ are halving cycles, with fixed point set $\si_\la^\C(F_\bullet^\C)$.
	
	Since we are no longer in the orientable case, in order to use the generalized Borel-Haefliger theorem, we need that the Grassmannians are Poincar\'e duality spaces and that $\fd(X)=2\fd(X^\Ga)$. Both of these follow from the Cartan description (Appendix \ref{subsec:Cartanmodel}): any flag manifold is a Poincar\'e duality space, and $\fd(X)=4k(n-k)$ and $\fd(X^\Ga)=2k(n-k)$. Then one can apply the generalized Borel-Haefliger theorem as in the proof of Theorem \ref{thm:doubleflagcirclespace}.
	
	{Finally, that the natural inclusions induce the isomorphisms can be shown as follows: the characteristic classes are mapped into each other via the natural inclusions, so there is a system of generators mapped into a system of generators with the same relations. For example, $$i:\Gr_{2k}(\R^{2n})\inj \Gr_{2k}(\R^{2n+1})$$ 
		pulls back $i^*p_j(S^{2n+1})=p_j(S^{2n})$ and $i^*p_j(Q^{2n+1})=p_j(Q^{2n})$ (even though the pullback $i^*Q^{2n+1}=Q^{2n}\oplus \ep$.) }
\end{proof}

\begin{proposition}\label{prop:realGrassmannianodd}
	The structure constants of $[\si_{D\la}]$ and $[\si_{L(D\la)}]$ in $H^*(\Gr_{2k+1}(\R^{2n}))$ are completely determined by the Littlewood-Richardson structure constants of $[\si_{D\la}]$ (Corollary \ref{cor:realLR}) and
	$$ [\si_{D\la}]\cdot [\si_{L0}]=[\si_{L(D\la)}],\qquad [\si_{L0}]^2=0.$$
\end{proposition}
\begin{proof}
	Since $[\si_{L0}]\in H^{2n-1}$ lives in odd degree, and multiplication is graded commutative, $[\si_{L0}]^2$ is 2-torsion, therefore zero rationally.
	
	To show $[\si_{D\la}]\cdot [\si_{L0}]=[\si_{L(D\la)}]$, we use the following lemma.
\end{proof}
\begin{lemma}
	In $\Gr_{k}(\R^n)$, for appropriate transverse flags $E_\bullet,F_\bullet$ there exists a flag $G_\bullet$, such that
	\begin{equation}\label{eq:LSchubert}
	\si_{L0}(E_\bullet)\cap \si_{\la}(F_\bullet)=\si_{L\la}(G_\bullet)
	\end{equation}
\end{lemma}
\begin{proof}
	Let $E_\bullet$ be the standard flag, $F_\bullet$ the opposite flag
	$$F_\bullet= \bra e_{n}\ket\leq \bra e_{n}, e_{n-1}\ket\leq \ldots \leq \bra e_{n},e_{n-1},\ldots e_{2}\ket  \leq \bra e_{n}, e_{n-1},\ldots, e_{1}\ket$$ 
	and $G_\bullet$ be obtained from $F_\bullet$ by exchanging $e_1$ and $e_{n}$:
	$$G_\bullet= \bra e_{1}\ket\leq \bra e_{1}, e_{n-1}\ket\leq \ldots \leq \bra e_{1},e_{n-1},\ldots e_{2}\ket  \leq \bra e_{1}, e_{n-1},\ldots,e_3,e_2, e_{n}\ket.$$ 
	
	$E_\bullet$ and $F_\bullet$ are transverse flags, which imply that the Schubert varieties $	\si_{L0}(E_\bullet)\cap \si_{\la}(F_\bullet)$ intersect transversely. The rank conditions defining $\si_{L0}(E_\bullet)$ translate to 
	$$U\in \si_{L0}\acsa E_1\leq U\leq E_{n-1}$$
	in particular, $\si_{L0}(E_\bullet)$ is a subGrassmannian $\Gr_{k}(\R^{n-2})$. The following observation allows us to conclude: if $E_1\leq U\leq E_{n-1}$, then
	$$ \dim(U\cap F_j)=k\acsa \dim(U\cap G_j)=k+1$$
	By comparing the rank conditions defining $\si_{\la}$ and $\si_{L\la}$ the two sides of \eqref{eq:LSchubert} are equal.
\end{proof}
\begin{remark}
	\begin{itemize}
		\item[i)] The previous lemma was stated for arbitrary $k$ and $n$; a simple verification shows that the subGrassmannian $\si_{L0}$ is always coorientable and therefore a cycle. However, the computation which shows that it appears in every incidence relation with zero coefficient only holds if $k$ is odd and $n$ is even, otherwise $[\si_{L0}]$ is a 2-torsion element.
		\item[ii)] If one defines the usual $\U(1)$-action on $\R^{2n}$ by identifying it with $\C^n$, the induced action on $\Gr_{2k+1}(\R^{2n})$ has no fixed points; indeed, $\C^n$ has no real odd dimensional invariant subspaces. $\Gr_{2k}(\R^{2n+1})$ has zero Euler characteristic so it cannot be a circle space.
	\end{itemize}
\end{remark}
\section{Integer coefficients and Steenrod squares}\label{sec:integer}

It is a classical result that all torsion in $H^*(\Gr_k(\R^\infty);\Z)$ is of order 2, see e.g.\ \cite[Theorem 24.7]{Borel1967}. By a theorem of Ehresmann \cite[p.\ 81]{Ehresmann1937}, this also holds in the case of finite Grassmannians. Therefore the rational and mod 2 reductions of an integer cohomology class completely determine its value. In particular, one can express the integer classes of Schubert cycles in terms of Pontryagin classes and Bockstein's of Stiefel-Whitney classes, see Theorem \ref{thm:integerSchubert} below. In this chapter, we extend Ehresmann's theorem to the case of even real flag manifolds $\Fl_{2\D}$, and conjecture that it holds in general:
\begin{theorem}\label{thm:2torsion}
	For even real flag manifolds $2\Tor(H^*(\Fl_{2\D};\Z))=0$, i.e.\ all torsion of $H^*(\Fl_{2\D};\Z)$ is of order 2. In particular, if $\be$ denotes the Bockstein homomorphism and $H^*_{\operatorname{free}}$ denotes the free part:
	$$ H^*(\Fl_{2\D};\Z)=H^*_{\operatorname{free}}(\Fl_{2\D};\Z)\bigoplus \im \be$$
\end{theorem}

For a general space $X$, one way to compute $H^*(X;\Z)$ is by computing the mod $p$ Bockstein spectral sequences. Degeneration of the mod 2 Bockstein SS on the $E_2$-page is equivalent to saying that the 2-primary part of $H^*(X;\Z)$ consists entirely of elements of order 2. The differential $d_1$ on the first page $E_1^{p,q}=H^{p+q}(X;\F_2)$ of the mod 2  Bockstein spectral sequence is the first Steenrod square $\Sq^1$. 
In order to compute the $E_2$-page, we will give a combinatorial description of the $\Sq^1$-action on the additive basis of Schubert cycles. We will then use this description to compute the $E_2$-page of the Bockstein SS.

The cohomology groups $H^*(\Fl_{2\D};\Z)$ then can be completely determined by using this theorem, the Cartan description of Section \ref{subsec:Cartanmodel} and the universal coefficient theorem. This is a combinatorial computation which can be carried out similarly as was done in \cite{He2017} for the case of Grassmannians. It relies on the following observation: if all torsion in $H^*(X;\Z)$ is of order two, then $P_{\Tor}=\frac{t}{t+1}(P_2-P_0)$, where $P_{\Tor}$ denotes the Poincar\'e polynomial of the ranks of $\Z_2$'s in $H^*(X;\Z)$, $P_2$ and $P_0$ denote the mod 2 and rational coefficient Poincar\'e polynomials of $X$.
\begin{proposition}\label{prop:ranks}
	Let $\D=(d_1\stb d_r)$, $N=\sum d_i$. The Poincar\'e polynomial of the ranks of $\Z_2$'s in $H^*(\Fl_{2\D}^\R;\Z)$ is given by
	$$ P_{\Tor}(\Fl_{2\D})=\frac{t}{t+1}\left(\frac{\prod_{j=1}^{2N} (1-t^j)}{\prod_{i=1}^r \prod_{j=1}^{2d_i}(1-t^j)}-\frac{\prod_{j=1}^N (1-t^{4j})}{\prod_{i=1}^r \prod_{j=1}^{d_i}(1-t^{4j})}\right).$$
\end{proposition}

{\subsection{Steenrod squares}
Lenart \cite{Lenart1998} computed the \emph{Steenrod coefficients} $c^k_{\la\mu}\in \F_2$ in complex Grassmannians:
$$\Sq^k[\si_\la]=\sum_{\mu }c^k_{\la\mu}[\si_\mu]$$
As he remarks, the Steenrod coefficients are obstructions for the triviality of the attaching maps: if $c^k_{\la\mu}$ is nonzero, then the corresponding attaching map is nontrivial. He also gives an example \cite[Section 6]{Lenart1998}, where the converse does not hold. The more general case of Steenrod coefficients $c_{IJ}^k$ for arbitrary flag manifolds of simply connected, semisimple Lie groups (and arbitrary $p$) was determined by Duan and Zhao \cite{DuanZhao2007}. Notice that the distinction of real and complex is irrelevant; by results of \cite{FranzPuppe2006} and \cite{VanHamel2007}:
$$ \Sq^{2k}[\si_I^\C]=\sum_{J}c^{k}_{IJ}[\si_J^\C]\acsa \Sq^{k}[\si_I^\R]=\sum_{J }c^k_{IJ}[\si_J^\R].$$

On the other hand, for $k=1$, the Steenrod coefficients $c^1_{IJ}$ for real flag manifolds vanish if and only if the incidence coefficients $[\Om_I,\Om_J]$ vanish:

\begin{proposition}\label{prop:Steenrod}
	Let $\Om_J\leq \Om_I\se \Fl_\D(\R^N)$ of neighboring dimensions, $I, J\in \OSP(\D)$. Set $c_{IJ}:=|[\Om_I,\Om_J]/2|\in \F_2$ where $[\Om_I,\Om_J]$ are the incidence coefficients, described in \eqref{eq:incidence}. Then
	$$ \Sq^1[\si_I]=\sum c_{IJ} [\si_J].$$
\end{proposition}
For a related statement for the case of Grassmannians in the algebraic geometric context (Chow-Witt ring) see \cite[Theorem 1.1]{Wendt2018}. The Proposition follows from the following classical viewpoint on Bockstein homomorphisms (see e.g.\ \cite{MosherTangora1968}):
\begin{lemma}\label{lemma:Bockstein}
	If $(C,d)$ is a cochain complex, then the Bockstein homomorphism of $(C,2d)$ 
	$$\be:H^*(C\otimes\F_2,2d\otimes \F_2=0)\to H^{*+1}(C,2d)$$ 
	satisfies
	$$ \be[\rho z]=[dz]\in H^{*+1}(C,2d),$$
	for all $z\in C$, where $\rho:C\to C\otimes \F_2$ is the mod 2 reduction.
\end{lemma}
The proof is an unraveling of the definition of the Bockstein homomorphism.
\begin{proof}[Proof of Proposition]
	Apply the Lemma to the Vassiliev complex (all coefficients are $\pm2$) and use the classical fact that the mod 2 reduction of the Bockstein homomorphism is $\Sq^1$.
\end{proof}

Thus the knowledge of the incidence coefficients \eqref{eq:incidence} yields a combinatorial formula for the Steenrod coefficients of $\Sq^1$ in an arbitrary partial flag manifold in terms of Schubert cycles.

\begin{remark}
	{There is a geometric proof of Proposition \ref{prop:Steenrod}. This involves the Thom-Wu theorem which states that if $i:Z\inj X$ is a smooth submanifold, then $\Sq^k[Z]=i_! w_k(\nu_i)$. Together with the description of $[\Om_I,\Om_J]$ in Proposition \ref{incidenceasbundle}, this proves Proposition \ref{prop:Steenrod}.}
\end{remark}

\subsection{Integral cycles}
Before proving Theorem \ref{thm:2torsion}, we deduce some easy consequences about integral cycles. 
\begin{proposition}
	Let $X$ be a manifold satisfying $2\Tor(H^*(X;\Z))=0$. Then every integral cycle $[Z]\in H^*(X;\Z)$ is one of the two following types
	\begin{itemize}
		\item $[Z]\in \im \be$, in which case $2[Z]=0$ or
		\item $[Z \mod 2]\neq 0\in \ker\Sq^1/\im \Sq^1$, in which case $[Z]\neq 0$ rationally.
	\end{itemize}
	
\end{proposition}
In particular, by Lemma \ref{lemma:Bockstein}, the 2-torsion part of $H^*(\Fl_{2\D};\Z)$ is generated by cycles of the form $[d\Om_I/2]$.

We relate Schubert cycles to integral characteristic classes. Denote a Stiefel-Whitney monomial in $H^*(\Fl_{\D}(\R^N);\F_2)$ by
$$w_M=\prod_{i,j}w_{i}(D_j)^{m_{ij}},$$
where $M=(m_{ij})\in \N^{N}$ denotes a multiindex ($N=\sum_{i=1}^rd_i$). Then every element of $H^*(\Fl_{\D}(\R^N);\F_2)$ can be written as a sum of $w_M$'s (not necessarily uniquely!), write $\sum_{} w_M$ for a generic element. Then $\sum w_M$ is the mod 2 reduction of a $2$-torsion class in $H^*(\Fl_\D;\Z)$ iff $\sum w_M\in\im \Sq^1$. For such elements, denote by $\sum v_M$ the unique second order element in $\im \be \leq H^*(\Fl_{2\D};\Z)$ whose mod 2 reduction equals $\sum w_M$. 

A consequence of Theorem \ref{thm:2torsion} is that from the mod 2 and rational descriptions of double Schubert classes $[\si_{DI}]$, one can express the integer classes $[\si_{DI}]$ in terms of characteristic classes. Using the previous notations, we can give a formula:

\begin{theorem}\label{thm:integerSchubert} 
	The integral classes $[\si_{DI}]\in H^*(\Fl_{2\D}(\R^{N});\Z)$ equal
	$$ [\si_{DI}]=s_{I}(p_{i}(D_j))+\{s_{DI}(v_{i,j})-s_{I}(v_{i,2j}^2)\},$$
	$j=1\stb r$ and $i=1\stb d_j$ where $s_I(x_{i,j})$ denotes the Schubert polynomial in the variables $x_{i,j}$ (in terms of elementary symmetric polynomials \cite[Section 5]{BernsteinGelfandGelfand1973}). 
	
	Here $s_{DI}(w_{i}(D_j))-s_I(w_{2i}(D_j)^2)\in H^*(\Fl_{2\D};\F_2)$ is in $\im \Sq^1$, and therefore lifts to a unique second order element denoted $\{s_{DI}(v_{i,j})-s_{I}(v_{i,2j}^2)\}\in H^*(\Fl_{2\D};\Z)$.
\end{theorem}
\begin{proof}
	The rational class is $ [\si_{DI}]=s_I(p_{i,j})$ by the generalized Borel-Haefliger theorem (Corollary \ref{cor:realBGG}) and the mod 2 class is $s_{DI}(w_{i,j})$ by the classical mod 2 Borel-Haefliger theorem. The theorem then follows from the fact that the mod 2 reduction of the $j$th integral Pontryagin class $p_j$ is $w_{2j}^2$.
\end{proof}

\begin{remark}{
	Contrary to what the notation might suggest, $\{s_{DI}(v_{i,j})-s_{I}(v_{i,2j}^2)\}$ is \emph{not} a polynomial in some classes $v_{i,j}$; it is an element, whose mod 2 reduction is the polynomial $s_{DI}(w_{i,j})-s_I(w_{i,2j}^2)$. Determining the class $[\si_{DI}]$ in terms of multiplicative generators of $H^*(\Fl_{2\D};\Z)$ is a challenge we will not consider.
}
\end{remark}
\subsection{Proof of Theorem \ref{thm:2torsion}}
In the rest of this section we prove Theorem \ref{thm:2torsion}, but first we recall some generalities about the Bockstein spectral sequence and the Bockstein cohomology of a space.
\subsubsection{Bockstein cohomology}
The \emph{Bockstein cohomology} of a topological space $X$ can be defined as follows: since $\Sq^1:H^*(X;\F_2)\to H^{*+1}(X;\F_2)$ satisfies $\Sq^1\circ \Sq^1=0$, one can regard $H^*(X;\F_2)$ as a chain complex and compute its cohomology $$H_{\be}^*(X):=H^*(H^*(X;\F_2);\Sq^1).$$ 
Since $\Sq^1$ is the differential $d_1$ on the page $E_1^{p,q}=H^{p+q}(X;\F_2)$ of the mod 2 Bockstein spectral sequence, so  $H_{\be}^*(X)$ is just the $E_2$-page of the Bockstein spectral sequence. Denote by $P_{\be}(X)$ the Poincar\'e polynomial of $H_{\be}^*(X)$ and by $P_0(X)$ the Poincar\'e polynomial of $H^*(X;\Q)$. We will use the following Proposition of Borel-Hirzebruch, \cite{BorelHirzebruch1959} Lemma 30.4 and 30.5 (1), cf.\ also \cite{Borel1967}:
\begin{proposition}\label{prop:BorelSqcohomology}
	If $H^*(X;\Z)$ is finitely generated, then the 2-primary component of $H^*(X;\Z)$ consists of elements of order 2 (i.e.\ the Bockstein spectral sequence degenerates) if and only if
	$$ P_{\be}(X)=P_0(X).$$
\end{proposition}
They show that in the case of infinite Grassmannians $G_N=\Gr_{N}(\R^\infty)$
\begin{equation}\label{eq:Steenrodcohomologyinfinite}
	P_0(G_N)=P_{\be}(G_N)=\prod_{i=1}^{\lfloor \frac N 2\rfloor}\frac{1}{1-t^{4i}}.
\end{equation}
In order to compute $P_{\be}$, $\F_2[w_1\stb w_{N}]$ is decomposed into the tensor product of subalgebras $A_i$ each invariant under $\Sq^1$. The Poincar\'e polynomial $P_{\be}$ is then the product of the Poincar\'e polynomials of $A_i$: $P_{\be}=\prod_{i} P_{A_i}$. The generators of $H_{\be}^*$ are also identified to be $[w_{2i}^2]$. 

We will reduce the finite case to the infinite case, so let us discuss the infinite version of flag manifolds. Given $\D=(d_1\stb d_r)$,
let 
\begin{equation}\label{eq:Dm}
	\D^{m}=(d_1\stb d_{r-1},d_r+m),
\end{equation} 
$m$ possibly infinite. Note that $\Fl_{\D^\infty}$ is the classifying space of the parabolic subgroup $P=\GL(d_1\stb d_{r-1})$ (using the notation of Section \ref{subsec:realflaggeometry}). 
\begin{proposition}\label{prop:PSq}
	$$P_{\be}(\Fl_{2\D^\infty})=\prod_{i=1}^{r-1}\prod_{j=1}^{d_i}\frac{1}{1-t^{4j}}=P_0(\Fl_{2\D^\infty})$$
	and $H_{\be}^*(\Fl_{2\D^\infty})$ is generated by $[w_{2j}^2(D_i)]$. In particular, all torsion of $H^*(\Fl_{2\D^\infty};\Z)$ is of order 2.
\end{proposition}
\begin{proof}
	Using that 
	$$ H^*(\Fl_{2\D^\infty})=\F_2[w_j(D_i):i=1\stb r-1,j=1\stb d_i],$$
	this follows from \eqref{eq:Steenrodcohomologyinfinite}, similarly to Borel's computations \cite[pp.\ 85--86]{Borel1967}.
\end{proof}

In the finite case, matters are complicated by the relations, since such a decomposition into subalgebras does not exist. This is why we will use the additive (Schubert basis) description of the cohomology groups, together with the combinatorial description of $\Sq^1$ (Proposition \ref{prop:Steenrod}).

\subsubsection{Stabilization of Schubert classes}
Next we discuss stabilization of Schubert classes. With the notation $\D^m$ introduced in \eqref{eq:Dm}, we have direct sum maps (see Section \ref{subsec:directsum}):
$$\K_m:\Fl_{\D}(\R^N)\inj \Fl_{\D^{m}}(\R^{N+m}). $$
The current indexing of the Schubert cycles does not satisfy nice stability properties, so from now on we will change the convention of indexing the Schubert varieties $\si_I$, $I\in \OSP(\D)$. Let $\,^\vee:\OSP(\D)\to \OSP(\D)$ be the involution defined as follows: for $I\in \OSP(\D)$, replace each element $k\in I_j$ by $N+1-k$:
$$ I^\vee_j:=\{N+1-k:k\in I_j\}$$
for all $j$, where $|\D|=N$. Let $\sii_I:=\si_{I^\vee}$ and $\Omm_I:=\Om_{I^\vee}$ for all $I\in \OSP(\D)$. For complete flag manifolds $\D=(1^N)$, this indexing convention agrees with \cite{Fulton1992}, see also \cite[p.\ 20]{FultonPragacz}. With the new conventions $\ell(\I)=\codim \sii_\I$ (for the notation $\ell(\I)$ see \eqref{eq:ellI}) and $\Omm_\I$ and $\Omm_\J$ are adjacent if $\J$ is obtained from $\I$ by exchanging $a\in \I_\al \leftrightarrow b\in \I_\be$, $a<b$, $\al<\be$. The incidence relations are (for the notation $T_I(a,b)$ see \eqref{eq:GreaterLessTangentNormal}):
\begin{equation}\label{eq:modifiedincidence} [\Omm_\I,\Omm_\J]=\begin{cases}
0,\qquad &T_\I(a,b) \text{ even}\\
\pm 2,\qquad &T_\I(a,b) \text{ odd}\\
\end{cases}
\end{equation}
With this indexing convention, Schubert classes mod 2 satisfy the following stabilization property (this can be shown by a transversality argument):
\begin{equation}\label{eq:KmSchubert}
 \K_m^*[\sii_{\J}]=\begin{cases}
[\sii_\I]\qquad &\text{if } \J=\I^m\\
0 \qquad &\text{else,}
\end{cases}
\end{equation}
where for $I\in \OSP(\D)$, define
$$I^m:=(I_1,I_2\stb I_{r-1},I_r\cup\{N+1\stb N+m\}).$$
In particular,
\begin{equation}\label{eq:kerK}
\ker \K_m^*=\bra [\sii_\I]:\exists k>N:k\not\in \I_r\ket.
\end{equation}
\begin{remark}
	For $\I\in \OSP(\D)$, stabilization implies that there exists $s_\I\in H^*(\Fl_{\D^\infty};\F_2)$, such that $K_\infty^*s_\I=[\sii_{\I}]$ and factors through $[\sii_{\I^m}]$ for all $m$. Since $\Fl_{\D^\infty}$ is the classifying space of $\GL(d_1\stb d_{r-1})$, $s_\I\in H^*(\Fl_{\D^\infty})$ is a polynomial $q(w_i^j)$ in the universal Stiefel-Whitney classes $w_i^j$, and since $\K_\infty^*w_i^j=w_{i}(D_j)$, stabilization provides universal formulas $[\sii_{\I^m}]=q(w_i(D_j))$ in terms of characteristic classes (the minimal Schubert polynomials, see \cite[Section 5]{BernsteinGelfandGelfand1973}).
\end{remark}
Stabilization also implies that there exists a universal formula for Steenrod squares in terms of Schubert cycles: there exist $c^k_{\I\J}\in \F_2$, such that
$$ \Sq^k s_\I=\sum_{}c^k_{\I\J} s_\J,\qquad \Sq^k [\sii_{\I^m}]=\sum_{}c^k_{\I\J} [\sii_{\J^m}]$$
holds for all $m$. Indeed, there are only finitely many $s_\I$ in $H^*(\Fl_{\D^\infty};\F_2)$ of fixed codimension, and Steenrod operations are natural. 
\subsubsection{Bockstein cohomology of even flag manifolds}
Now we start computing Bockstein cohomology using Proposition \ref{prop:Steenrod}. Let 
$$ Z_m:=\ker(\Sq^1:H^*(\Fl_{\D^{m}})\to H^{*+1}),\qquad B_m:=\im(\Sq^1:H^{*-1}(\Fl_{\D^{m}})\to H^{*}),$$
both graded $\F_2$-vector spaces, let $H_m:=Z_m/B_m$, $q_m:Z_m\to H_m$ and denote $Z:=Z_0$ and $B:=B_0$, $H:=H_0$, $q:=q_0$. 
In the following we compare $Z_m$ with $Z$ and $B_m$ with $B$ using the description \eqref{eq:kerK}. Naturality of $\Sq^1$ implies that
\begin{itemize}
	\item[i)] $B=\K^*_mB_m$, 
	\item[ii)] $Z\supseteq \K^*_mZ_m$.
\end{itemize}		
These imply that $q\circ \K_m^*:Z_m\to Z\to H$ induces a map $\hat{\K}^*_m:H_m\to H$ satisfying
\begin{equation}\label{eq:kerKm}
	\ker \hat{\K}_m^*=q_m(\ker \K_m^*).
\end{equation} 

\begin{lemma}\label{lemma:ZZ1}
	Let $\D=(2d_1\stb 2d_r)$. For even flag manifolds $\Fl_{\D}$: If $[\sii_\I]\in Z$, then $[\sii_\I]\in \K_1^*Z_1$.
\end{lemma}
\begin{proof}
Assume $[\sii_\I]\in Z$ for some $\I\in \OSP(\D)$. Write $$ \Sq^1[\sii_{\I^1}]=\sum_{\J\in T}[\sii_{\J}]\in H^*(\Fl_{\D^1}),$$
where by assumption $[\sii_\J]\in \ker \K_1^*$ for all $\J\in T\se \OSP(\D^1)$. If $[\sii_\J]\in \ker \K_1^*$, then by \eqref{eq:kerK} $\J$ is obtained from $\I^1$ by exchanging $N+1\in \I^1_{r}$  with some $a\in \I_\al^1$ which satisfies \begin{equation}\label{eq:swap}
	a=\max \I_\al^1>\max \I_{\al+1}^1>\ldots >\max \I_{r-1}^1,
\end{equation}
otherwise the number of inversions changes by at least 2. To determine $[\Om_I,\Om_J]$, we use \eqref{eq:modifiedincidence}, recall the notations of \eqref{eq:GreaterLessTangentNormal}. Clearly, $G_{\I^1}(N+1,\al-1,r-1)=0$. By \eqref{eq:swap}, $$L_{\I^1}(a,\al,r)=\sum_{i=\al+1}^{r-1}2d_i$$ so by \eqref{eq:modifiedincidence}, $[\Om_\I,\Om_\J]=0$. Hence the Steenrod coefficients also vanish by Proposition \ref{prop:Steenrod}, therefore $T$ is empty, proving the lemma.
\end{proof}

\begin{corollary}\label{cor:Steenrodcohomologyfinite}
	For even flag manifolds $\Fl_{2\D}$:
	$$Z=\K^*_\infty Z_\infty,\qquad H=\hat{\K}^*_\infty H_\infty,$$
\end{corollary}
\begin{proof}
Let $\I\in \OSP(\D)$ and $\J\in \OSP(\D^m)$. A simple computation involving the number of inversions shows that $\I^m$ and $\J$ are adjacent in $\Fl_{\D^m}$ iff $\J=K^{m-1}$ for some $K\in \OSP(\D^1)$ and $\I^1$ and $K$ are adjacent in $\Fl_{\D^1}$. Therefore $[\sii_{\I^1}]\in Z_1$ iff $[\sii_{\I^m}]\in Z_m$, so  $Z=\K_1^*Z_1=\K_m^*Z_m$ for all $m$. This also implies surjectivity of $\hat{\K}_\infty^*$. 
\end{proof}

Now we can finally prove the theorem:
\begin{proof}[Proof of Theorem \ref{thm:2torsion}]	
	It follows from results of Borel's thesis that $\Tor(H^*(\Fl_{\D};\Z))$ is 2-primary (see \cite[Propositions 29.1, 30.1]{Borel1953}), so we can concentrate on $p=2$. By Corollary \ref{cor:Steenrodcohomologyfinite}:
	$$ P_{\be}(\Fl_{2\D})=P_{\be}(\Fl_{2\D^\infty})-P(\ker \hat{\K}_m^*)$$
	Let $Q\leq Z_m$ be a subspace such that $q:Q\isoto H_m$. Then  the Poincar\'e polynomials satisfy
	$$ P(\ker \hat{\K}_m^*)=P(Q\cap \ker \K_m^*).$$
	By Proposition \ref{prop:PSq}, such a subspace is given by $Q=\F_2[w_{2j}(D_i)^2]$, so by \eqref{eq:KmSchubert}
	\begin{equation}\label{eq:kerSq}
		Q\cap \ker \K_\infty^*=\bra s_\J(w_{2j}^2):\J\neq \I^\infty, \I\in \OSP(2\D)\ket.
	\end{equation}
	
	Denote the rational coefficient pullback induced by $\K_\infty$ by
	$$\K_0^*:H^*(\Fl_{2\D^\infty};\Q)\to H^*(\Fl_{2\D};\Q)$$
	Since $\K_0^*$ is also surjective,
	$$ P_{0}(\Fl_{2\D})=P_{0}(\Fl_{2\D^\infty})-P(\ker \K_0^*)$$
	and 
	\begin{equation}\label{eq:kerK0}
	\ker \K_0^*=\bra s_\J(p_j):\J\neq \I^\infty, \I\in \OSP(2\D)\ket.
	\end{equation}
	Comparing \eqref{eq:kerSq} and \eqref{eq:kerK0}, $P(\ker\K_0^*)=P(\ker \hat{\K}_m^*)$, and since $P_0(\Fl_{2\D^\infty})=P_{\be}(\Fl_{2\D^\infty})$ by Proposition \ref{prop:PSq}, 
	$P_0(\Fl_{2\D})=P_{\be}(\Fl_{2\D})$ and we can conclude by Proposition \ref{prop:BorelSqcohomology}.
\end{proof}
}

\begin{remark}
	In the case of Grassmannians $\Gr_{K}(\R^N)$, the previous proof can also be adapted, and one can show that $P_{\be}= P_0$. Lemma \ref{lemma:ZZ1} also holds, except when $K$ is odd and $N$ even, when $Z\su Z_1$ also contains the $L$-Schubert classes. In the case of arbitrary flag manifolds $\Fl_{\D}$ these computations appear to be more complicated; it is no longer true that the Bockstein cohomology generators are Schubert varieties $[\sii_I]$, but sums of them, see Appendix \ref{sec:tables}.
\end{remark}
\begin{conjecture}\label{conj:2torsion}
	For any real flag manifold (more generally any $R$-space) all torsion of $H^*(\Fl_\D;\Z)$ is of order 2. 
\end{conjecture}
\begin{remark}
	If this conjecture holds, then the cohomology groups $H^*(\Fl_\D^\R;\Z)$ can be determined from the Universal Coefficient Theorem as in \cite{He2017}, see also Proposition \ref{prop:ranks}.
\end{remark}

\section{Enumerative applications: lower bounds}\label{subsec:enumerative}

The cohomology ring structure in terms of Schubert classes gives information about enumerative geometric \emph{Schubert problems}:

Given generic complete flags $F_\bullet^1\stb F_\bullet^r$ in $\R^N$, what is the cardinality of
$$ \left|\bigcap_{j=1}^r \si_{\la_j}(F_\bullet^j)\right|=?$$
for $\si_j:=\si_{\la_j}$ of total dimension $\Gr_{2k}(\R^{2n})$?

The word generic is a subtle point here: we will say that the flags are \emph{generic}, if the corresponding Schubert varieties are transversal. This is an open condition in the configuration space by the Kleiman-Bertini theorem. The main property of generic configurations $\mathcal{G}$ relevant to us is that the number of solutions is locally constant on $\mathcal{G}$. 

In the complex case $\F=\C$, a Schubert problem can be solved by multiplying Schubert cycles: since everything is complex, at a smooth transversal intersection all tangent spaces have canonical orientations, therefore all intersections come with the same sign. Therefore the cohomology product of $[\si_{j}]$ is an element $n[*]$ of $H^{top}(\Gr_{k}(\C^n))\iso \Z\bra [*] \ket$, and this number $n$ is the answer to the Schubert problem.

In the real case $\F=\R$, there are no canonical orientations, therefore each transversal intersection $p$ comes with a sign, depending on whether the orientation of the tangent spaces $T_p\si_j$ agrees with the orientation of the tangent space $T_p\Gr_k(\R^n)$. Therefore the cohomological calculation only gives a signed sum of the points, hence a lower bound to the Schubert problem. The actual number of solutions depends on the configuration (the choice of the flags $F_\bullet^i$), and there is a range of numbers that might appear as the number of solutions. This range is not known in general, and it has been subject to extensive examination, see \cite{Sottile1997}, \cite{SoprunovaSottile2006}, \cite{HeinSottileZelenko2016}.  An infinite series of examples has been computed via elementary methods in the case of Grassmannians in \cite{FeherMatszangosz2016}. For example, in $\Gr_8(\R^{16})$ the number of solutions to the Schubert problem $\si_\la^4$ for $\la=(4,4,4,4)$ can be $\{6,14,30,70\}$, see \cite{FeherMatszangosz2016}. Recently, similar problems were considered for Grassmannians using Chow-Witt rings \cite{Wendt2018}.\\

The dependence of the number of solutions on the given configuration has the following explanation. In the complex case, the singular configurations form an at least one complex codimensional subvariety of the configuration space, so the space of nonsingular configurations is connected. In the real case, the singular configurations can be one \emph{real} codimensional, in which case the configuration space falls apart into connected components (\emph{chambers}).\\

An upper bound for the range is given by the number of solutions for the corresponding generic complex Schubert problem. Here some caution is required when discussing genericity: one has to show that there exist real generic flags which are complex generic when regarded as complex flags. Indeed, this is the case: the subset of complex nongeneric configurations can be defined by real equations, so there exist real flags $F_\bullet^j$ which are complex generic. For such flags, all intersections of $ \si_j^\C(F_\bullet^j)$ are transverse, therefore so are those of $\si_j^\R(F_\bullet^j)$, so such configurations are also real generic.\\

It is a natural question, whether a real enumerative problem is \emph{maximal/fully real} \cite{Sottile1997}, i.e.\ whether there exists a configuration for which the number of solutions agrees with the number of solutions for the same complex problem. This is true for real Schubert problems in Grassmannians as shown in \cite{Vakil2006}. \\

Another natural question in real enumerative geometry, is to find a lower bound for the range of solutions \cite{Welschinger2005}, \cite{HeinHillarSottile2013}. {See \cite{Kollar2015} for examples of enumerative problems with a strict lower bound of zero, i.e.\ no real solutions.} As we have already mentioned, for Schubert calculus the cohomology calculation also gives a lower bound. By the description of the real Littlewood-Richardson coefficients (Corollary \ref{cor:realLR}) we have:
\begin{proposition}\label{prop:realSchubert}
	The number of solutions of a double real Schubert problem $(D\la_j)$ is bounded below by the number of solutions to the half sized complex one $(\la_j)$.
\end{proposition}

This means that any complex Schubert problem can be ``doubled''. We illustrate this via two examples.

\begin{example}
 How many real 6-planes in $\R^{12}$ intersect 9 given 6-planes in at least 2 dimensions?
 
 This problem can be written as $\Yboxdim5pt\yng(2,2)^9$ in $\Gr_{6}(\R^{12})$. 
 By the previous Proposition, a lower bound is given by the half-sized complex problem, namely $\Yboxdim5pt\yng(1)^9$ in $\Gr_3(\C^6)$. 
 This in turn is equal to the degree of $\Gr_3(\C^6)$ via the Pl\"ucker embedding, which is 42; this is known in general, see e.g.\ \cite[p.\ 247]{Harris1995}. So the answer is 42: there are at least 42 such real 6-planes.
\end{example}

\begin{example}
	How many real 8-planes in $\R^{14}$ intersect 6 given 8-planes in at least 4 dimensions?

This problem is

\begin{minipage}{\textwidth}
\centering
	\begin{picture}(50,50)(2,2)

\put(0,50){\line(0,-1){40}}
\put(10,50){\line(0,-1){40}}
\put(20,50){\line(0,-1){40}}
\put(0,50){\line(1,0){20}}
\put(0,40){\line(1,0){20}}
\put(0,30){\line(1,0){20}}
\put(0,20){\line(1,0){20}}
\put(0,10){\line(1,0){20}}

\put(23,50){\small $6$}
\end{picture}
\end{minipage}

 in $\Gr_{8}(\R^{14})$. Again, a lower bound is given by the half-sized complex problem, namely $\Yboxdim5pt\yng(1,1)^6$ 
 in $\Gr_4(\C^7)$. A computation using the Pieri rule and duality shows that the answer to the complex problem is 
 $\Yboxdim5pt\yng(1,1)^6=14$. 
 So the real problem has at least 14 solutions.
\end{example}

As we have seen, the cohomology of real Grassmannians (and even flag manifolds) with integer coefficients contains only $2$-torsion. If in a Schubert problem all Schubert varieties are cycles, but one of them is $\Z_2$-torsion, then the corresponding cycles multiply to zero in cohomology (at least if the flag manifold is orientable), which is a trivial lower bound. Note however that the corresponding Schubert problem can be, and usually is nontrivial. There exist enumerative problems, which do not have any cohomological interpretation: the corresponding Schubert varieties are not cycles. For example, dual transversal Schubert varieties always intersect in one point, and not all Schubert varieties are cycles. Summarizing: for the purpose of obtaining enumerative lower bounds, we don't lose anything by working with rational coefficient cohomology.\\

Alternatively, considering $\F_2$-coefficient cohomology, one can apply the original Borel-Haefliger theorem \cite{BorelHaefliger1961} to Grassmannians to obtain mod 2 information about a Schubert problem: 
\begin{proposition}
	The number of solutions of a real Schubert problem has the same parity as the number of solutions of the corresponding complex Schubert problem.
\end{proposition}
For certain Schubert problems, one can say more than mod 2 congruence of the solutions, see e.g.\ \cite[Theorem 5.7]{FeherMatszangosz2016} and \cite{HeinSottileZelenko2016}. We conclude with a conjecture:
\begin{conjecture}
	The lower bound of Proposition \ref{prop:realSchubert} is sharp.
\end{conjecture}

We make some remarks. Take real complete flags $F_\bullet^j\in \Fl(\R^n)$, such that $F_{2i}^j$ are $\U(1)$-invariant. Then the set of solutions 
$$S:=\bigcap_j\si^\R_{DI_j}(F_{\bullet}^j)$$ 
is a $\U(1)$-invariant subset. If $S$ is finite, then each point $W\in S$ is a $\U(1)$-fixed point, i.e.\ complex. Therefore, it is a solution to the corresponding half sized complex Schubert problem 
$$ W\in S_\C=\bigcap_j\si^\C_{I_j}(F_{\bullet,\C}^j).$$

This reduces the conjecture to one about genericity: Does there exist a complex generic configuration of flags $F^j_{\bullet,\C}$, which as real flags $F^j_{\bullet,\R}$ are real generic (for the double sized real problem)?

\appendix

\section{Topology: Cartan model}\label{subsec:Cartanmodel}
A pair $(G,K)$ of compact connected Lie groups is a \emph{Cartan pair}, if $K\leq G$ and $H_G^*$ is a polynomial ring on $\rk G$ many generators, $\rk G- \rk K$ many of which restrict to zero via $\rho^*:H_G^*\to H_K^*$. 

For Cartan pairs, there is a simple description of $H^*(G/K;\Q)$, due to Cartan and Borel \cite{Borel1953}, see also \cite{Terzic2011} for a summary.
\begin{theorem}[Borel, Cartan]
	For a Cartan pair $(G,K)$
	$$ H^*(G/K;\Q)\iso H_K^*/(\im\rho^*)_+\bigotimes \bigwedge[x_{r_i-1}]_{i=p+1}^n $$
	where $r_i$ are the degrees of the polynomial generators restricting to zero via $\rho^*:H_G^*\to H_K^*$, $n=\rk G$, $p=\rk K$, $\deg x_j=j$ and $(W)_+$ denotes the ideal generated by elements of positive degree of the subspace $W$.
\end{theorem}

Let $\SO(\D):=\prod_{i=1}^m \SO(d_i)$, $\operatorname{O}(\D):=\prod_{i=1}^m \operatorname{O}(d_i)$ $N:=\sum d_i$. 
\begin{proposition}
	$(G,K_0)=(\SO(N),\SO(\D))$ is a Cartan pair for all $\D$.
\end{proposition}
\begin{proof}
	One has
	$$ \rho^*(p_*(S))=\prod_{i=1}^m p_*(S_i),$$
	where $S\to BSO(N)$ and $S_i\to BSO(d_i)$ denote the tautological bundles and $p_*$ denotes the total Pontryagin class. Let $n=\lfloor \frac N 2\rfloor$ be the rank of $\SO(N)$ and $q=\sum_{i=1}^m \lfloor \frac{d_i} 2\rfloor$ be the rank of $\SO(\D)$. Since $p_{\operatorname{top}}(S_i)=p_{\lfloor \frac{d_i}{2}\rfloor}(S_i)$, by examining degrees, one sees that $\rho^*(p_j(S))=0$ for $j>q$.
\end{proof}
The real flag manifold is a homogeneous space $\Fl_\D^\R=G/K$, where $G=SO(N)$ and $K=S(\operatorname{O}(\D))$. Since $K$ is not connected, the Borel-Cartan model cannot be directly applied. $H^*(G/K;\Q)$ was recently determined by He in \cite{He2019}, we state the theorem, also to be found in \cite{thesis}. 

Let $\D_0=(\lfloor d_1/2\rfloor,\lfloor d_2/2\rfloor\stb \lfloor d_m/2\rfloor)$. Since there is a covering $\Ga\to G/K_0\to G/K$ for $\Ga=\Z_2^{m-1}$, $H^*(G/K)=(H^*(G/K_0))^\Ga$. The $\Ga$-action acts by multiplying the Euler classes in $H_K^*$ by -1 and trivially on the Pontryagin classes and on the exterior algebra, for further details see \cite{He2019} or \cite{thesis}. 
\begin{theorem}\label{thm:realflagcohomologyCartan}
	Let $n=\lfloor \frac N 2\rfloor$ be the rank of $\SO(N)$ and $q=\sum_{i=1}^m \lfloor \frac{d_i} 2\rfloor$ be the rank of $\SO(\D)$, then
	$$ H^*(\Fl_\D^\R;\Q)\iso 
	H^{2*}(\Fl_{\D_0}^\C;\Q)\otimes \bigwedge[y_i]_{i=q+1}^n
	$$
	where $y_i=x_{4i-1}$ except if $N$ even, $y_n=x_{N-1}$. $H^{2*}$ means that the degrees are doubled and $\deg x_j=j$. The first term $H^{2*}$ is generated as an algebra by $p_i(D_j)$ with the relations given by the identity $\prod_{j=1}^m p_*(D_j)=1$.
\end{theorem}
%

\section{Schubert cycles: the general case of partial flag manifolds}
\label{sec:tables}
Using the coefficients of the Vassiliev complex described in Section \ref{sec:realflagcohomology}, we computed some of the generators of $H^*(\Fl^\R_\D;\Q)$ in terms of Schubert cells using SageMath's homology package \cite{sagemath}. We include below some cases which are not covered by Theorems \ref{thm:Schubertcycles} and \ref{thm:realGrassmannianadditive}, in particular, complete, odd and other examples. As previously mentioned, for general $\Fl_{\D}$ it is no longer true that rational cohomology classes can be represented by Schubert varieties, but a signed sum of Schubert cells. The choice of the generators are not unique, as we already saw on the example of $\Fl(\R^4)$, Section \ref{subsec:Fl4}. In all the examples we have computed, the coefficients of these Schubert cells are $\pm 1$. In the tables, we use the following conventions.

\textbf{Notation.} The Schubert cells $\Om_I\se \Fl_\D$ are parametrized by ordered set partitions $I\in\OSP(\D)$. These are denoted by the one-line notation of their minimal representatives: elements of $I_j$ are listed in increasing order, $I_j$ and $I_{j+1}$ is separated by a comma. The $+$ sign separates the Schubert cells whose sums are generators. We do not keep track of the sign of the cells, as this can vary according to convention (even though relative to each other, the signs do make sense). In the last table, the ordered set partitions are elements of $1,2,\ldots,11$; for typographical reasons 10 and 11 are preceded by a space. 
\subsection{The complete case}
The two extreme cases of real flag manifolds are Grassmannians and complete flag manifolds. We understand the Schubert calculus of Grassmannians by Propositions \ref{prop:realGrassmannianeven} and \ref{prop:realGrassmannianodd}. 
For the case of complete flag manifolds, the answer appears to be less simple, as we illustrate in Tables \ref{tab:completeflag} and \ref{tab:completeflag7}. For the case of $\Fl(\R^3)$ see also \cite[p. 5]{Kocherlakota1995} and for $\Fl(\R^4)$, see also \cite[p. 529]{CasianStanton1999}. 
\subsection{The odd case}
There are two cases when flag manifolds $\Fl_\D^\R$ are orientable; all $d_i$ are even or all odd. If all $d_i$ are even, we understand the generators of rational cohomology by Theorem \ref{thm:Schubertcycles}. If all $d_i$ are odd, the answer again appears to be less simple, see Tables \ref{tab:flag333} and  \ref{tab:flag335} (and also the complete cases). 
\subsection{The other cases}
See Table \ref{tab:flag234} for a nonorientable case.\\

These examples hopefully illustrate that although there is a simple description of the cohomology of real flag manifolds in terms of topology (cf.\ Cartan model, Appendix \ref{subsec:Cartanmodel}), in general there is some nontrivial combinatorics involved in translating that description to the Schubert calculus setting.

\begin{table}
	\centering
	\begin{tabular}{|c|c|c|c|c|}
		\hline	
		$\deg$	& $\Fl(\R^3)$ 	&	$\Fl(\R^4)$ & 	$\Fl(\R^5)$ & 	$\Fl(\R^6)$ \\
		\hline
		&$\La[x_3]$		& 	$\La[x_3,y_3]$&	$\La[x_3,x_7]$& $\La[x_3,x_7,y_5]$ \\
		\hline				
		0 & 321  		&	4321 		&	54321				&	654321						\\ \hline
		3 & 123 		&	2341+4123	&	34521+52341+54123	&	456321+634521+652341+654123 \\ \hline
		3 &  			&	3214		&						&								\\ \hline
		5 &				&				&						&	365214						\\ \hline
		6 & 			&	1234		&						&								\\ \hline
		7 & 			&				&	14325				&	432561+632145				\\ \hline
		8 &				&				&						&	345216+523416+541236		\\ \hline
		10 & 			&				&	1234				&	234561+236145+412563+612345 \\ \hline
		12 & 			&				&						&	125436						\\ \hline
		15 &			&				&						&	123456						\\ \hline				
	\end{tabular}
	\caption{Sums of Schubert cells generating $H^*(\Fl(\R^n);\Q)$, labeled by permutations $S_n$}\label{tab:completeflag}
\end{table}
	\begin{table}
	\centering
	\begin{tabular}{|c|c|}
		\hline	
		$\deg$	& $\Fl(\R^7)$ \\ \hline
		&$\La[x_3,x_7,x_{11}]$ \\ \hline				
		0 & 7654321\\ \hline
		3 &  5674321+7456321+7634521+7652341+7654123 \\ \hline
		7 & 5436721+7432561+7632145 \\ \hline
		10 & 3456721+3472561+3672145+5236741+5416723+5436127+\\ 
		& +5632147+7234561+7236145+7412563+7612345 \\ \hline
		11 & 1476325\\ \hline
		14 & 1456327+3416527+5216347+5412367\\ \hline
		18 & 1236547 \\ \hline
		21 & 1234567\\ \hline				
	\end{tabular}
	\caption{Sums of Schubert cells generating $H^*(\Fl(\R^7);\Q)$, labeled by permutations $S_7$}\label{tab:completeflag7}
\end{table} 
\begin{table}
	\centering
	\begin{tabular}{|c|c|}
		\hline	
		$\deg$	& $\Fl_{333}$ \\
		\hline
		&$H^*(\Fl_{222};\Q)\otimes \La[x_{15}]$\\
		\hline				
		0 & 789,456,123	\\ \hline
		4 & 789,236,145	\\ \hline
		4 & 569,478,123	\\ \hline
		8 &	349,678,125+369,458,127+389,256,147+589,234,167	\\ \hline
		8 & 569,238,147+589,234,167	\\ \hline
		12 & 349,258,167	\\ \hline
		15 & 167,258,349	\\ \hline
		19 & 167,234,589	\\ \hline
		19 & 145,278,369+147,256,389	\\ \hline
		23 & 123,478,569	\\ \hline				
		23 & 145,236,789	\\ \hline				
		27 & 123,456,789	\\ \hline
	\end{tabular}
	\caption{Sums of Schubert cells generating $H^*(\Fl_{333};\Q)$, labeled by $\OSP(3,3,3)$}\label{tab:flag333}
\end{table}
\pagestyle{empty}
	\begin{table}
		\centering
		\begin{tabular}{|c|c|}
			\hline	
			$\deg$	& $\Fl_{234}$ \\ \hline
			& $H^*(\Fl_{224};\Q)$\\ \hline				
			0 & 89,567,1234\\ \hline
			4 & 89,347,1256\\ \hline
			4 & 67,589,1234\\ \hline
			8 & 89,127,3456\\ \hline
			8 & 45,789,1236\\ \hline
			8 & 47,569,1238+67,349,1258\\ \hline
			12 & 45,369,1278\\ \hline
			12 & 23,789,1456\\ \hline				
			12 & 27,369,1458+67,129,3458\\ \hline
			16 & 23,569,1478\\ \hline
			16 & 25,349,1678+45,129,3678\\ \hline
			20 & 23,149,5678\\ \hline
		\end{tabular}
		\caption{Sums of Schubert cells generating $H^*(\Fl_{234};\Q)$, labeled by $\OSP(2,3,4)$}\label{tab:flag234}
	\end{table}


	\begin{table}
		\centering
		\begin{tabular}{|c|c|}
			\hline	
			$\deg$	& $\Fl_{335}$ \\ \hline
			& $H^*(\Fl_{224};\Q)\otimes \La[x_{19}]$\\ \hline				
			0 & 9 10 11,678,12345\\ \hline
			4 & 9 10 11,458,12367 \\ \hline
			4 & 78 11, 69 10,12345\\ \hline
			8 & 9 10 11, 238, 14567\\ \hline
			8 & 56 11,89 10,12347\\ \hline
			8 & 58 11,67 10,12349+78 11,45 10,12369\\ \hline
			12 & 56 11,47 10,12389\\ \hline
			12 & 34 11,89 10,12567\\ \hline				
			12 & 38 11,47 10,12569+78 11,23 10,14569 \\ \hline
			16 & 34 11,67 10,12589\\ \hline
			16 & 36 11,45 10,12789+56 11,23 10,14789\\ \hline
			19 & 189,27 10,3456 11 \\ \hline
			20 & 34 11,25 10,16789 \\ \hline
			23 & 127,89 10,3456 11+167,29 10,3458 11\\ \hline
			23 & 169,25 10,3478 11+189,256,347 10 11\\ \hline
			27 & 167,258,349 10 11\\ \hline
			27 & 125,69 10,3478 11+145,29 10,3678 11\\ \hline
			27 & 149,23 10,5678 11+189,234,567 10 11\\ \hline		
			31 & 125,678,349 10 11+145,278,369 10 11\\ \hline		
			31 & 123,49 10,5678 11\\ \hline		
			31 & 147,238,569 10 11+167,234,589 10 11 \\ \hline
			35 & 145,236,789 10 11\\ \hline
			35 & 123,478,569 10 11\\ \hline		
			39 & 123,456,789 10 11\\ \hline	
		\end{tabular}
		\caption{Sums of Schubert cells generating $H^*(\Fl_{335};\Q)$, labeled by $\OSP(3,3,5)$}\label{tab:flag335}
	\end{table}

\clearpage
\newpage

\bibliographystyle{plain}
\bibliography{biblio}
\end{document}